\documentclass[11pt,a4paper]{article} 
\usepackage{amsfonts, amsmath, amsthm, amssymb, graphicx, hyperref, xcolor}
\hypersetup{colorlinks, linkcolor={black!90!black}, citecolor={blue!50!black}, urlcolor={blue!80!black}}
\usepackage[british]{babel}
\usepackage[a4paper, margin=3cm]{geometry} 
\usepackage{tikz}
\theoremstyle{plain}
\newtheorem{lemma}{Lemma}[section]
\newtheorem{theorem}{Theorem}[section]
\newtheorem{proposition}{Proposition}[section]
\newtheorem{corollary}{Corollary}[section]

\title{Interlacement limit of a stopped random walk trace on a torus}
\author{Antal A. J\'{a}rai \and Minwei Sun}
\date{\today}

\def\P{\mathbf{P}}
\def\E{\mathbf{E}}
\def\Var{\mathrm{Var}}
\def\Z{\mathbb{Z}}

\def\T{\mathbb{T}}

\def\cL{\mathcal{L}}

\def\cpty{\mathrm{Cap}}

\def\bfY{\mathbf{Y}}
\def\bfx{\mathbf{x}}
\def\bfy{\mathbf{y}}
\def\bfK{\mathbf{K}}

\def\eps{\varepsilon}
\def\es{\emptyset}

\begin{document}
\maketitle

\begin{abstract}
We consider a simple random walk on $\Z^d$ started at the origin and stopped on its 
first exit time from $(-L,L)^d \cap \mathbb{Z}^d$. 
Write $L$ in the form $L = m N$ with $m = m(N)$ and $N$ an integer going to infinity 
in such a way that $L^2 \sim A N^d$ for some real constant $A > 0$.
Our main result is that for $d \ge 3$, the projection of the stopped trajectory 
to the $N$-torus locally converges, away from the origin, to an interlacement process
at level $A d \sigma_1$, where $\sigma_1$ is the exit time of a Brownian motion 
from the unit cube $(-1,1)^d$ that is independent of the interlacement process.
The above problem is a variation on results of Windisch (2008) and Sznitman (2009).
\end{abstract}



\section{Introduction}
\label{sec:intro}
A special case of a result of Windisch \cite{Windisch2008} --- extended further in \cite{JiriAugusto2016} 
--- states that the trace of a simple random walk on the discrete $d$-dimensional torus
$(\mathbb{Z} / N \mathbb{Z})^d$, for $d\ge 3$, started from stationarity and run for time 
$u N^d$ converges, in a local sense, to an interlacement process at level $u$, 
as $N \to \infty$. In this paper we will
be concerned with a variation on this result, for which our motivation was a 
heuristic analysis of an algorithm we used to simulate high-dimensional 
loop-erased random walk and the sandpile height distribution \cite{JaraiSun2019a}. 
Let us first describe our main result and then discuss the motivating 
problem. 

Consider a discrete-time lazy simple random walk $(Y_t)_{t \ge 0}$ starting at the origin $o$ on $\mathbb{Z}^d$. 
We write $\P_o$ for the probability measure governing this walk.
We stop the walk at the first 
time $T_L$ it exits the large box $(-L,L)^d$, where $L$ is an integer. We will take $L=L(N)$ of the form
$L = m N$, where $m = m(N)$ and $N$ is an integer, such that 
$L^2 \sim A N^d$ for some $A \in (0,\infty)$, as $N \to \infty$.
We consider the projection of the trajectory $\{ Y_t : 0 \le t < T_L \}$ to 
the $N$-torus $\mathbb{T}_N = [ -N/2, N/2 )^d \cap \mathbb{Z}^d$.
The projection is given by the map $\varphi_N : \Z^d \to \T_N$,
where for any $x \in \mathbb{Z}^d$, $\varphi_N(x)$ is the 
unique point of $\mathbb{T}_N$ such that $\varphi_N(x) \equiv x$ $\pmod N$, where 
congruence $\pmod{N}$ is understood coordinate-wise. 

%so that the projected trajectory covers on the order of $N^d$ points in the 
%torus. Our main result (Theorem \ref{thm:main} below) says that for $d \ge 3$ 
%and a finite set $K \subset \mathbb{Z}^d$ that is not `too close' to the 
%origin, the asymptotic probability (as $N \to \infty$) that the trace 
%$\{ \varphi(Y_t) : 0 \le t < T \}$ does not meet $\bfK$ is given by the 
%probability that an interlacement process, considered at a random level, 
%does not meet $K$. The random level is given by the standard 
%Brownian exit time $\sigma$ from $[-1,1]^d$ multiplied by $c A$ for some constant 
%$c = c(d)$. 

Let $\sigma_1$ denote the exit time from $(-1,1)^d$ of a standard Brownian motion started at $o$.
We write $\mathbb{E}$ for the expectation associated to this Brownian motion.
For any finite set $K \subset \Z^d$, let $\cpty(K)$ denote
the capacity of $K$ \cite{LawlerLimic2010}.
For any $0 < R < \infty$ and $x \in \mathbb{Z}^d$, we denote
$B_R(x) = \{ y \in \mathbb{Z}^d : |y - x| < R \}$, where $| \cdot |$ is the Euclidean norm.
Let $\mathcal{K}_R$ denote the collection of all subsets of $B_R(o)$. 
Given $\mathbf{x} \in \mathbb{T}_N$ let $\tau_{\mathbf{x},N} : \mathbb{T}_N \to \mathbb{T}_N$ 
denote the translation of the torus by $\mathbf{x}$. Let $g : \mathbb{N} \to (0, \infty)$ 
be any function satisfying $g(N) \to \infty$.

\begin{theorem}
\label{thm:main}
Let $d \ge 3$.
For any $0 < R < \infty$, any $K \in \mathcal{K}_R$, and any $\mathbf{x}$ satisfying 
$\tau_{\mathbf{x},N} \varphi_N(B_R(o)) \cap \varphi_N(B_{g(N)}(o)) = \es$ we have
\begin{equation}
\label{e:main}
 \mathbf{P}_o \left[ \varphi_N(Y_t) \not\in \tau_{\mathbf{x},N} \varphi_N(K),\, 0 \le t < T_L \right]
 = \mathbb{E} \left[ e^{-d A \sigma_1 \cpty(K)} \right] + o(1) 
   \quad \text{as $N \to \infty$.}
\end{equation}
The error term depends on $R$ and $g$, but is uniform in $K$ and $\mathbf{x}$.
\end{theorem}

Note that the trace of the lazy simple random walk stopped at time $T$ is the 
same as the trace of the simple random walk stopped at the analogous exit time.
We use the lazy walk for convenience of the proof.

Our result is close in spirit --- but different in details ---
compared to a result of Sznitman \cite{Sznitman2009}
that is concerned with simple random walk on a discrete
cylinder. The \emph{interlacement process} was introduced by Sznitman in 
\cite{Sznitman2010}. It consists of a one-parameter family 
$(\mathcal{I}^u)_{u > 0}$ of random subsets of $\Z^d$ ($d \ge 3$), 
where the distribution of $\mathcal{I}^u$ can be characterized by the relation
\begin{equation}
\label{e:interl-def}
  \P [ \mathcal{I}^u \cap K = \es ]
   = \exp ( - u \cpty(K) ) \quad \text{for any finite $\es \not= K \subset \Z^d$.}
\end{equation}
The precise construction of a process satisfying \eqref{e:interl-def} represents $\mathcal{I}^u$ as the trace of a Poisson cloud of bi-infinite random walk trajectories (up to time-shifts), where $u$ is an
intensity parameter. We refer to \cite{Sznitman2010} and the 
books \cite{Drewitz2014,Sznitman-book} for further details.
Comparing \eqref{e:main} and \eqref{e:interl-def} we now formulate precisely what we mean by 
saying that the stopped trajectory, locally, is described by an interlacement process
at the random level $u = A d \sigma_1$.

Let $g' : \mathbb{N} \to (0, \infty)$ 
be any function satisfying $g'(N) \to \infty$. 
Note this does not have to be the same function as $g(N)$.
Let $\bfx_N$ be an arbitrary sequence satisfying $\tau_{\mathbf{x}_N,N} \varphi_N(B_{g'(N)}(o)) \cap \varphi_N(B_{g(N)}(o)) = \es$.
Define the sequence of random configurations $\omega_N \subset \Z^d$ by 
\begin{equation*}
    \omega_N = \{ x \in \Z^d : \tau_{\mathbf{x}_N,N} \varphi_N(x) \in \{ \varphi_N(Y_t) : 0 \le t < T_L \} \}.
\end{equation*}
Define the process $\tilde{\mathcal{I}}$ by requiring that for all finite $K \subset \Z^d$ we have
\begin{equation*}
    \mathbf{P} [ \tilde{\mathcal{I}} \cap K = \es ] = \mathbb{E} [ e^{-dA\sigma_1 \cpty(K)}]. 
\end{equation*}
To see that this formula indeed defines a process that is also unique, write the right-hand side as
\begin{equation*}
    \int_0^\infty e^{-u \cpty(K)} \, f_{\sigma_1}(u) \, du
    = \int_0^\infty \mathbf{P} [ \mathcal{I}^u \cap K = \es ] \, f_{\sigma_1}(u) \, du,
\end{equation*}
where $f_{\sigma_1}$ is the density of $A d \sigma_1$. Then via the inclusion-exclusion formula, we see that we necessarily have for all finite sets $B \subset K$ the equality
\begin{equation*}
    \mathbf{P} [ \tilde{\mathcal{I}} \cap K = B]
    = \int_0^\infty \mathbf{P} [ \mathcal{I}^u \cap K = B ] \, f_{\sigma_1}(u) \, du,
\end{equation*}
and the right-hand side can be used as the definition of the finite-dimensional marginals of $\tilde{\mathcal{I}}$. Note that $\tilde{\mathcal{I}}$ lives in a compact space (the space can be identified with $\{ 0, 1\}^{\Z^d}$ with the product topology). Hence the finite-dimensional marginals uniquely determine the distribution of $\tilde{\mathcal{I}}$, by Kolmogorov's extension theorem.

\begin{theorem}
\label{thm:convergence}
Let $d \ge 3$. Under $\mathbf{P}_o$ the law of the configuration $\omega_N$ converges weakly to the law of $\tilde{\mathcal{I}}$, as $N \to \infty$.
\end{theorem}

\begin{proof}[Proof of Theorem \ref{thm:convergence} assuming Theorem \ref{thm:main}.]

For events of the form $\{ \omega_N \cap K = \es \}$, Theorem \ref{thm:main} immediately implies that \[ \P_o [ \omega_N \cap K = \es ] \stackrel{N \to \infty}{\longrightarrow} 
   \P [ \tilde{\mathcal{I}} \cap K = \es ].
   \]
For events of the form $\{ \omega_N \cap K = B \}$, the inclusion-exclusion formula represents $\P_o [ \omega_N \cap K = B ]$ as a linear combination of probabilities of the former kind, and hence convergence follows.
\end{proof}

Our motivation to study the question in Theorem \ref{thm:main} was a
simulation problem that arose in our numerical study of high-dimensional 
sandpiles \cite{JaraiSun2019a}. We refer the interested reader 
to \cite{Redig2006-survey,Dhar2006-survey,Jarai2018-survey} for background on sandpiles. 
In our simulations we needed to generate loop-erased random walks (LERW) from the origin $o$ to the 
boundary of $[-L,L]^d$ where $d \ge 5$. The LERW is defined by running a
simple random walk from $o$ until it hits the boundary, and erasing all loops 
from its trajectory chronologically, as they are created. We refer to the
book \cite{LawlerLimic2010} for further background on LERW (which is not needed 
to understand the results in this paper). It is known from 
results of Lawler \cite{Lawler2013} that in dimensions $d \ge 5$ the LERW visits 
on the order of $L^2$ vertices, the same as the simple random walk 
generating it. As the number of vertices visited is a lot smaller than the
volume $c L^d$ of the box, an efficient way to store the path generating
the LERW is provided by the well-known method of \emph{hashing}. We refer to \cite{JaraiSun2019a} for
a discussion of this approach, and only provide a brief summary here. Assign to any 
$x \in [-L,L]^d \cap \mathbb{Z}^d$ an integer value 
$f(x) \in \{ 0, 1, \dots, C L^2 \}$ that is used to label the information 
relevant to position $x$, where $C$ can be a large constant or slowly growing to infinity. 
Thus $f$ is necessarily highly non-injective. 
However, we may be able to arrange that with high probability the restriction 
of $f$ to the simple random walk trajectory is not far from injective, and 
then memory use can be reduced from order $L^d$ to roughly $O(L^2)$.

A simple possible choice of the hash function $f$ can be to compose the map 
$\varphi_N : [-L,L]^d \cap \mathbb{Z}^d \to \mathbb{T}_N$ with a linear 
enumeration of the vertices of $\mathbb{T}_N$, whose range has the required
size\footnote{This is slightly different than what was used in \cite{JaraiSun2019a}.}.
The method can be expected to be effective, if the projection $\varphi_N(Y[0,T))$
spreads roughly evenly over the torus $\mathbb{T}_N$ with high probability.
Our main theorem establishes a version of such a statement, as the right-hand side
expression in \eqref{e:main} is independent of $\mathbf{x}$.

%Briefly explain the strategy of the main proof here

We now make some comments on the proof of Theorem \ref{thm:main}.
We refer to \cite[Theorem 3.1]{Drewitz2014} for the strategy of the proof in the case when
the walk is run for a fixed time $u N^d$. The argument presented there goes by decomposing
the walk into stretches of length $\lfloor N^\delta \rfloor$ for some $2 < \delta < d$ , 
and then estimating the (small) probability in each stretch that $\tau_{\bfx,N} \varphi_N(K)$ is 
hit by the projection. We follow the same outline for the stopped lazy random walk. 
However, the elegant time-reversal argument given in \cite{Drewitz2014} is not convenient in our
setting, and we need to prove a delicate estimate on the probability that $\tau_{\bfx,N} \varphi_N(K)$ 
is hit, \emph{conditional} on the start and end-points of the stretch.  
For this, we only want to consider stretches with 
``well-behaved'' starting and end-points. We also classify stretches as ``good stretch'' where
the total displacement is not too large, and as ``bad stretch'' otherwise. We do this in such a way that the 
expected number of ``bad stretches'' is small and summing over the ``good stretches'' 
gives us the required behaviour. 

\emph{Possible generalisations.} \\
(1) It is not essential that we restrict to the simple 
random walk: any random walk for which the results in Section \ref{sec:preliminaries}, hold
(such as finite range symmetric walks) would work equally well. \\
(2) The paper \cite{Windisch2008} considers several distant sets $K^1, \dots, K^r$, and we believe this 
would also be possible here, but would lead to further technicalities in the presentation. \\
(3) It is also not essential that the rescaled domain be $(-1,1)^d$, and we believe it could be
replaced by any other domain with sufficient regularity of the boundary.

\emph{A note on constants.} All constants will be positive and finite. Constants denoted $C$ or $c$ 
will only depend on dimension $d$ and may change from line to line. If we need to refer to a constant later, 
it will be given an index, such as $C_1$.

We now describe the organization of this paper. 
%Section \ref{sec:intro} introduces the main result and the motivating problem.
In Section \ref{sec:preliminaries}, we first introduce some basic notation, then we 
recall several useful known results on random walk and state the key propositions required 
for the proof of the main theorem, Theorem \ref{thm:main}.
Section \ref{sec:proof-of-main} contains the proof of the main theorem, assuming 
the key propositions.
Finally, in Section \ref{sec:proof of key propositions} we provide the proofs of the propositions 
stated in Section \ref{sec:preliminaries}.

\section{Preliminaries}
\label{sec:preliminaries}

\subsection{Some notation}
\label{ssec:some notaion}

We first introduce some notation used in this paper. 
In Section \ref{sec:intro}, we denoted the discrete torus 
$\T_N = [-N/2, N/2)^d \cap \mathbb{Z}^d$, $d\geq 3$ 
and the canonical projection map $\varphi_N: \Z^d \to \T_N$.
From here on, we will omit the $N$-dependence and write $\varphi$ and $\tau_{\bfx}$ instead.

We write vertices and subsets of the torus in bold, i.e.~$\bfx \in \T_N$ and 
$\bfK \subset \T_N$. In order to simplify 
notation, in the rest of the paper we abbreviate $\bfK = \tau_{\bfx} \varphi(K)$.

% --- We replaced $K$ by $\varphi^{-1}(\bfK)$ and use $K$ only for general definitions. ---

%and let $K \subset \Z^d$ be a translate of $K$ with the property that $\varphi(K) = \bfK$.
% $\bfx \in \T_N$ and $\bfK = \varphi(K) \subset \T_N$.
%Let $(X_n)_{n\ge 0}$ be the simple random walk on $\Z^d$ and 

Let $(Y_t)_{t \ge 0}$ be a discrete-time lazy simple random walk on $\Z^d$, 
that is,
\[ \P [ Y_{t+1} = y' \,|\, Y_t = x' ]
   = \begin{cases}
	   \frac{1}{2} & \text{when $y' = x'$;} \\
		 \frac{1}{4d} & \text{when $|y' - x'| = 1$.}
	   \end{cases} \] 
We denote the corresponding lazy random walk on $\T_N$ by 
%$(\bfX_n)_{n\geq 0} = (\varphi(X_n))_{n\geq 0}$ and 
$(\bfY_t)_{t\geq 0} = (\varphi(Y_t))_{t\geq 0}$. 
Let $\P_{x'}$ denote the distribution of the lazy random walk on $\Z^d$ started from $x' \in \Z^d$,
and write $\P_{\bfx}$ for the distribution of the lazy random walk on $\T_N$ started from 
$\bfx = \varphi(x') \in \T_N$. We write $p_t(x',y') = \P_{x'} [ Y_t = y' ]$ for the 
$t$-step transition probability. Further notation we will use:
\begin{itemize}
\item $L = m N$, where $L^2 \sim A N^d$ as $N \to \infty$ for some constant $A \in (0,\infty)$
\item $D = (-m,m)^d$, rescaled box, indicates which copy of the torus the walk is in
\item $n = \lfloor N^\delta \rfloor$ for some $2 < \delta < d$, long enough for the mixing property on the
torus, but short compared to $L^2$
\item $\mathbf{x_0} \in \mathbf{K}$ is a fixed point of $\mathbf{K}$
\item we write points in the original lattice $\Z^d$ with a prime, such as $y'$, and 
decompose a point $y'$ as $y N+\bfy$ with $y$ in another lattice isomorphic to $\Z^d$ and $\bfy = \varphi(y') \in \T_N$
\item $T = \inf \{ t \ge 0 : Y_t \not\in (-L,L)^d \}$, the first exit time
from $(-L,L)^d$
%\item for $y \in \mathbb{Z}^d$, denote by $[[ y ]]$ the rounding of $y/N$ to the nearest
%lattice point in $D$
\item $S = \inf \{ \ell \ge 0 : Y_{n \ell} \not\in (-L,L)^d \}$, so that the first multiple of $n$
when the rescaled point $Y_{n \ell} / N$ is not in $(-m,m)^d$ equals $S\cdot n$
\end{itemize}
We omit the dependence on $d$ and $N$ from some notation above for simplicity.

\subsection{Some auxiliary results on random walk}
\label{ssec:auxiliary results}
%short introduction about 5 results LCLT, maximal inequality, Gaussian bounds, capacity and mixing times, then write in details.

In this section, we collect some known results required for the proof of Theorem \ref{thm:main}. 
We will rely heavily on the Local Central Limit Theorem (LCLT)
\cite[Chapter 2]{LawlerLimic2010}, with error term, 
and the Martingale maximal inequality \cite[Eqn.~(12.12) of Corollary 12.2.7]{LawlerLimic2010}. 
We will also use Equation (6.31) in \cite{LawlerLimic2010}, that relates $\cpty(K)$ to the
probability that a random walk started from the boundary of a large ball with radius 
$n$ hits the set $K$ before exiting the ball. %as $n\to\infty$.
In estimating some error terms in our arguments, sometimes we will use the
Gaussian upper and lower bounds \cite{HebischSaloff-Coste1993}. 
We also need to derive a lemma related to the mixing property on the torus 
\cite[Theorem 5.6]{LevinPeresWilmer2017} to show that the 
starting positions of different stretches are not far from uniform on the torus; 
see Lemma \ref{lem:mixing_property}.

We recall the LCLT from \cite[Chapter 2]{LawlerLimic2010}. 
%We used the same notations as in the book and adapted 
The following is a specialisation of \cite[Theorem 2.3.11]{LawlerLimic2010} 
to lazy simple random walk.
%%[page 11 LawlerLimic]
%For simple random walk in $\Z^d$,
%\begin{align*}
%\Gamma = d^{-1}I,
%\quad J^{*}(x) = d^{\frac{1}{2}}|x|,
%\quad J(x) = |x|.
%\end{align*}
%For lazy simple random walk in $\Z^d$,
The covariance matrix $\Gamma$ and the square root $J^{*}(x)$ of the associated quadratic 
form are given by  
\begin{align*}
\Gamma = (2d)^{-1}I,
\quad J^{*}(x) = (2d)^{\frac{1}{2}}|x|,
%\quad J(x) = \sqrt{2}|x|.
\end{align*}
where $I$ is the $(d \times d)$-unit matrix.

%[page 42 LawlerLimic]
Let $\bar{p}_t(x')$ denote the estimate of $p_t(x')$ that one obtains by 
the LCLT, for lazy simple random walk. We have  
\begin{align*}
\bar{p}_t(x') 
&= \frac{1}{(2\pi t)^{d/2}\sqrt{\det\Gamma}} \exp\left(-\frac{J^{*}(x')^2}{2t}\right) \\
&= \frac{1}{(2\pi t)^{d/2}(2d)^{-d/2}} \exp\left(-\frac{2d\,|x'|^2}{2t}\right) \\
&= \frac{\bar{C}}{t^{d/2}}\exp\left(-\frac{d\,|x'|^2}{t}\right).
\end{align*}
The lazy simple random walk $(Y_t)_{t\geq 0}$ in $\Z^d$ is aperiodic, 
irreducible with mean zero, finite second moment, and finite exponential moments. 
All joint third moments of the components of $Y_1$ vanish.

\begin{theorem}[\cite{LawlerLimic2010}, Theorem 2.3.11]
For lazy simple random walk $(Y_t)_{t\geq 0}$ in $\Z^d$,
%...decide which error estimate to use (additive or multiplicative)
there exists $\rho>0$ such that for all $t\geq 1$ and all $x'\in\Z^d$ with $|x'|<\rho t$,
\begin{align*}
p_t(x') 
= \bar{p}_t(x')\exp\left\{O\left(\frac{1}{t}+\frac{|x'|^4}{t^3}\right)\right\}.
\end{align*}
\end{theorem}

The Martingale maximal inequality in \cite[Eqn.~(12.12) of Corollary 12.2.7]{LawlerLimic2010} is stated as follows.
Let $(Y^{(i)}_t)_{t \ge 0}$ denote the $i$-th coordinate of $(Y_t)_{t \ge 0}$ ($1 \le i \le d$). The standard deviation 
$\sigma$ of $Y^{(i)}_1$ is given by $\sigma^2 = (2d)^{-1}$. For all $t \ge 1$ and all $r > 0$ we have
\begin{equation}
\label{e:maximal_ineq}
\P_o \left[\max_{0\leq j\leq t} Y^{(i)}_j \geq r\sigma\sqrt{t} \right]
\le e^{-r^2/2}\exp\left\{O\left(\frac{r^3}{\sqrt{t}}\right)\right\}.
\end{equation}
%where, in our case, $\sigma$ is the variance of a single coordinate for the lazy random walk, 
%$\sigma^2 = 1/2d$.

Now we state the result of \cite[Eqn.~(6.31)]{LawlerLimic2010}.
Recall that $B_r(o)$ is the discrete ball centred at $o$ with radius $r$. Let for any subset $B \subset \Z^d$
\[ \xi_B 
   = \inf \{ t \ge 1 : Y_t \not\in B) \}. \]
Let $\partial B_r(o) = \{ y' \in \Z^d \setminus B_r(o) : \text{$\exists x' \in B_r(o)$ such that $|x' - y'| = 1$} \}$. 
For a given finite set $K \subseteq\Z^d$, let $H_K$ denote the hitting time 
\[ H_K
   = \inf \{ t \ge 1 : Y_t \in K \}. \]
Then we have
\begin{equation}
\label{e:capacity_ball}
\frac{1}{2}\cpty(K)
= \lim_{r\to\infty}\sum_{y' \in \partial B_r(o)} \P_{y'}[H_{K} < \xi_{B_r(o)}].
\end{equation} 
Here $\cpty(K)$ is the capacity of $K$; see \cite[Section 6.5]{LawlerLimic2010}, 
which states the analogous statement for the simple random walk. Since we consider 
the lazy random walk, this introduces a factor of $1/2$.

In estimating some error terms in our arguments, sometimes we will use the
Gaussian upper and lower bounds \cite{HebischSaloff-Coste1993}: 
there exist constants $C = C(d)$ and $c = c(d)$ such that
\begin{equation}
\begin{split}
\label{e:gaussian_ub_lb}
  p_t(x',y') 
  &\le \frac{C}{t^{d/2}} \exp \left( - c \frac{|y'-x'|^2}{t} \right), \quad
		\text{for $x', y' \in \Z^d$ and $t \ge 1$;} \\
  p_t(x',y') 
  &\ge \frac{c}{t^{d/2}} \exp \left( - C \frac{|y'-x'|^2}{t} \right),
  \quad\text{for $|y'-x'| \le ct$.}	
\end{split}
\end{equation}
Recall that the norm $|\cdot|$ refers to the Euclidean norm.

Regarding mixing times, recall that lazy simple random walk on the $N$-torus mixes in time $N^2$ 
\cite[Theorem 5.6]{LevinPeresWilmer2017}.
% states that for the lazy random walk 
%on the $d$-dimensional torus $\T_N$, if $\eps<1/2$, then we have
%\begin{equation}
%\label{e:mixing_prop}
%t_{\mathrm{mix}}(\eps) \leq d\,N^2\,\log_4(d/\eps).
%\end{equation}
%Here $t_{\mathrm{mix}}(\eps)$ can be defined as the minimal time $t$ such that the total variation
%distance between the distribution of $Y_t$ (starting from a fixed vertex of the torus)
%and the uniform distribution on $\T_N$ is $< \eps$.
%By using \eqref{e:mixing_prop} we derive the following lemma.
With this in mind we derive the following simple lemma.

Recall that $2<\delta<d$ and $n=\lfloor N^{\delta} \rfloor$.
\begin{lemma}
\label{lem:mixing_property}
There exists $C=C(d)$ such that for any $N \geq 1$ and any $t \ge n$ we have 
\begin{equation*}
\P_{\mathbf{o}}[\bfY_t = \bfx] 
\le \frac{C}{N^d},
\quad \bfx \in \T_N.
\end{equation*}
%\begin{equation*}
%\P_{\mathbf{o}}[\bfY_t = \bfx] 
%= \frac{1}{N^d}\left(1+ O\left(e^{-c\,N^{\delta-2}}\right)\right),
%\quad \bfx \in \T_N.
%\end{equation*}
\end{lemma}

\begin{proof}
Using the Gaussian upper bound, the left-hand side can be bounded by
\begin{equation*}
\begin{split}
    \sum_{x \in \Z^d} p_t(o,xN+\bfx)
  &\le \frac{C}{t^{d/2}} \sum_{x \in \Z^d} \exp \left( - c \frac{|xN + \bfx|^2}{t} \right)
	\le \frac{C}{t^{d/2}} \sum_{x \in \Z^d} \exp \left( - c \frac{|xN|^2}{t} \right) \\
    &\le \frac{C}{t^{d/2}} \sum_{k=0}^{\infty} (k+1)^{d-1} \frac{t^{d/2}}{N^{d}}\exp(-ck^2) \\
    &\le \frac{C}{N^{d}} \sum_{k=0}^{\infty} (k+1)^{d-1} \exp(-ck^2) 
    \le \frac{C}{N^d}.
\end{split}
\end{equation*}
Here we bounded the number of $x$ in $\Z^d$ satisfying $k \sqrt{t}/N \le |x| < (k+1) \sqrt{t}/N$ by $C(k+1)^{d-1} t^{d/2}/ N^{d}$, where $k=0, 1, 2, \dots$.

% \begin{equation*}
% \begin{split}
%     \sum_{x \in \Z^d} p_t(o,xN+\bfx)
%   &\le \frac{C}{t^{d/2}} \sum_{x \in \Z^d} \exp \left( - c \frac{|xN + \bfx|^2}{t} \right)
% 	\le \frac{C}{t^{d/2}} \sum_{x \in \Z^d} \exp \left( - c \frac{|xN|^2}{t} \right) \\
% 	&\le \frac{C}{t^{d/2}} \sum_{r = 0}^\infty (r+1)^{d-1} \exp \left( - c \frac{r^2 N^2}{t} \right) \\
% 	&\le \frac{C}{t^{d/2}} \int_0^\infty (r+1)^{d-1} \exp \left( - c \frac{r^2 N^2}{t} \right) \, dr \\
% 	&= \frac{C}{t^{d/2}} \int_0^\infty (\frac{\sqrt{tu}}{N}+1)^{d-1} \exp \left( - c u \right) \, \frac{\sqrt{t}}{2 N \sqrt{u}} \, du \\
% 	&\le \frac{C}{N^d} \int_0^\infty \frac{(\sqrt{u}+1)^{d-1}}{\sqrt{u}} \exp \left( - c u \right) \, du 
%   = \frac{C}{N^d}.
% \end{split}
% \end{equation*}

\end{proof}

% \le \frac{C}{N^d} \int_0^\infty (\sqrt{u})^{d-2} \exp \left( - c u \right) \, du 
%   = \frac{C}{N^d}. }

%\begin{proof}
%By the mixing property \eqref{e:mixing_prop}, assume 
%$\eps = \frac{1}{N^d}\cdot e^{-c\,N^{\delta-2}} < 1/2$, we have
%\begin{align*}
%t_{\mathrm{mix}}(\eps) 
%\leq d\,N^2\log_4(d\,e^{c\,N^{\delta-2}}N^d)
%= d\,N^2\frac{c\,N^{\delta-2} + \log N^d + \log d}{\log(4)}
%\leq \frac{2dcN^{\delta}}{\log(4)}
%= N^{\delta},
%\end{align*}
%if we take $c=\log(4)/2d$ in the last step.
%
%By the definition of mixing time $t_{\mathrm{mix}}$, the maximal distance between the measure
%of the lazy random walk on $\T_N$, denoted by $\P_{\mathbf{o}}[\bfY_n = \cdot]$, and the
%uniform measure is at most $\eps$, i.e.
%\begin{equation*}
%\sum_{\bfx\in\T_N}\left|\P_{\mathbf{o}}[\bfY_n = \bfx]-\frac{1}{N^d}\right|
%\leq \eps.
%\end{equation*}
%Then we deduce that for any $\bfx$
%\begin{equation*}
%\left|\P_{\mathbf{o}}[\bfY_n = \bfx]-\frac{1}{N^d}\right|
%\leq \eps.
%\end{equation*}
%Therefore
%\begin{equation*}
%\frac{1}{N^d}-\eps
%\leq \P_{\mathbf{o}}[\bfY_n = \bfx]
%\leq \frac{1}{N^d}+\eps
%= \frac{1}{N^d}(1+N^d\eps)
%= \frac{1}{N^d}\left(1+ O\left(e^{-c\,N^{\delta-2}}\right)\right).
%\end{equation*}
%%The above constant $c$ might be different, but we use the same notation $c$ for simplicity.
%\end{proof}

\subsection{Key propositions}
\label{ssec:key proposition}

%Let $\sigma$ denote the exit time of Brownian motion from $[-1,1]^d$, started at the origin. 
%\begin{theorem}
%\label{thm:main}
%There exists a constant $c_1 = c_1(A)$ such that whenever the distance between %$\varphi^{-1}(\mathbf{K})$ 
%and the origin greater than $N^{\zeta}$, we have
%\begin{equation}
%\label{e:main}
% \lim_{N \to \infty} \mathbf{P}_o \left[ Y_t \not\in \varphi^{-1}(\mathbf{K}),\, 0 %\le t < T \right]
% = \mathbf{E} \left[ e^{-c_1 \sigma \cpty(K)} \right].  
%\end{equation}
%\end{theorem}
%In other words, the asymptotic probability that $\mathbf{K}$ is not visited by the walk projected to the torus is given by an interlacement process at an independent random level $c_1 \sigma$, due to the random time $T$.

%%independent of the interlacement process

In this section we state some propositions to be used in Section \ref{sec:proof-of-main} 
to prove Theorem \ref{thm:main}.
The propositions will be proved in Section \ref{sec:proof of key propositions}.

The strategy of the proof is to consider stretches of length $n$ of the walk, 
and estimate the small probability in each stretch that $\bfK$ is hit by the projection. 
What makes this strategy work is that we can estimate, conditionally on the starting and 
end-points of a stretch, the probability that $\bfK$ is hit, and this event asymptotically 
decouples from the number of stretches.
The number of stretches will be the random variable $S$. Since $n S \approx T$, and $T$ is 
$\Theta(N^d)$ in probability, we have that $S$ is $\Theta(N^d/n)$. In Lemma \ref{lem:E1_prelim} 
below we show a somewhat weaker estimate for $S$ (which suffices for our needs). 

The main part of the proof will be to show that during a fixed stretch, $\bfK$ is not hit 
with probability 
\begin{equation}
\label{e:not-hit-K}
  1 - \frac{1}{2} \cpty(\bfK) \frac{n}{N^d} (1+o(1)). 
\end{equation}
Heuristically, conditionally on $S$ this results in the probability
\[ \left( 1 - \frac{1}{2} \cpty(\bfK) \frac{n}{N^d} (1+o(1)) \right)^{S}
   \approx \exp \left( - \frac{1}{2} \cpty(\bfK) \frac{n}{N^d} S \right), \]
and we will conclude by showing that $(n/N^d) S$ converges in distribution 
to a constant multiple of the Brownian exit time $\sigma_1$.

The factor $n/N^d$ in \eqref{e:not-hit-K} arises as the expected time spent 
by the projected walk at a fixed point of the torus during a given stretch.
The capacity term arises as we pass from expected time to hitting probability.

For the above approach to work, we need a small probability event on which the number of stretches or end-points of stretches are not sufficiently well-behaved.
First, we will need to restrict to realizations where 
$(\sqrt{\log \log n})^{-1} (N^d/n) \le S \le \log N (N^d/n)$, which occurs with high probability as $N \to \infty$ (see Lemma \ref{lem:E1_prelim} below). 
Second, suppose that the $\ell$-th stretch 
starts at the point $y'_{\ell-1}$ and ends at the point $y'_{\ell}$, that is, $y'_{\ell-1}$ and $y'_{\ell}$ are realizations of 
$Y_{(\ell-1) n}$ and $Y_{\ell n}$. In order to have a good estimate of the probability that $\bfK$ is hit during this stretch, we will need to impose a few conditions on $y'_{\ell-1}$ and $y'_{\ell}$. One of these is that the displacement $|y'_{\ell} - y'_{\ell-1}|$ is not too large: we will require that for all stretches, it is at most $f(n) \sqrt{n}$, for a function to be chosen later that increases to infinity with $N$. We will be able to choose $f(n)$ of the form $f(n) = C_1 \sqrt{\log N}$ in such a way that this restriction holds for all stretches with high probability. A third condition we need to impose, that will also hold with high probability, is that $y'_{\ell-1}$ is at least a certain distance $N^\zeta$ from $\varphi^{-1}(\bfK)$ for a parameter $0 < \zeta < 1$ (this will only be required for $\ell \ge 1$, and is not needed for the first stretch starting with $y'_0 = o$). The reason we need this is to be able to appeal to \eqref{e:capacity_ball} to extract the $\cpty(\bfK)$ contribution, when we know that $\bfK$ is hit from a long distance (we will take $r = N^\zeta$ in \eqref{e:capacity_ball}). The larger the value of $\zeta$, the better error bound we get on the approach to $\cpty(\bfK)$. On the other hand, $\zeta$ should not be too close to $1$, because  
we want the separation of $y'_{\ell-1}$ from $\bfK$ to occur with high enough probability.
%hitting estimate \eqref{e:not-hit-K} we want would simply not hold, if $\varphi(y'_\ell)$ or $\varphi(y'_{\ell+1})$ is too close to $\bfK$.

The set $\mathcal{G}_{\zeta,C_1}$ defined below represents realizations of $S$ and the sequence $Y_{n}, Y_{2n}, \dots, Y_{Sn}$ satisfying the above restrictions. Proposition \ref{prop:good_stretches} below implies that these restrictions hold with high probability. 
%The definition involves
%a parameter $0 < \zeta < 1$, 
%We now specify the choice of $\zeta$.
First, we will need $2 < \delta < d$ to satisfy the inequality
\begin{equation}
\label{e:choice-delta}
  \delta - \frac{2\delta}{d} > d - \delta 
  \qquad \Leftrightarrow \qquad 
  2\delta > \frac{d^2}{d-1}.
\end{equation}
%$\delta - \frac{2\delta}{d} > d - \delta.$
This can be satisfied if $d \ge 3$ and $\delta$ is sufficiently close to $d$,
say $\delta = \frac{7}{8} d$. 
Since the left-hand side of the left-hand inequality in \eqref{e:choice-delta}
equals $(\delta/d)(d-2)$, we can subsequently choose $\zeta$ such that we also have
\begin{equation}
\label{e:choice-zeta}
  0 < \zeta < \frac{\delta}{d}, \qquad\qquad  
	\zeta(d - 2) > d - \delta. 
\end{equation}
%With $0 < \zeta < 1$ a parameter to be fixed later,  
%We fix $\zeta < \delta/d$, with $\zeta$ close enough to 
%$\delta/d$ so that we also have $\zeta (2 - d) < \delta - d$.

%We now further explain what we mean by `good' stretches. 
With the parameter $\zeta$ fixed satisfying the above, we now define:
\begin{equation}
\label{e:def_G}
\mathcal{G}_{\zeta,C_1}
= \left\{ \left( \tau, (y_\ell, \mathbf{y}_\ell)_{\ell=1}^\tau \right) : 
   \parbox{7cm}{$(\sqrt{\log \log n})^{-1} \frac{N^d}{n} \le \tau \le \log N \frac{N^d}{n}$; \\
	 $y_\ell \in D$, $\bfy_\ell \in \T_N \setminus B(\mathbf{x_0}, N^{\zeta})$ for
	 $1 \le \ell < \tau$; \\
	 $y_\tau \in D^c$ and $\bfy_\tau \in \T_N \setminus B(\mathbf{x_0}, N^{\zeta})$; \\
	 $|y'_\ell - y'_{\ell-1}| \leq f(n)n^{\frac{1}{2}}$ for 
	 $1 \le \ell \le \tau$} \right\},
\end{equation}
where $f(n) = C_1 \,\sqrt{\log N}$ and recall that 
we write $y'_\ell = y_\ell N + \bfy_\ell$ and $y'_{\ell-1} = y_{\ell-1} N + \bfy_{\ell-1}$,
and we define $y'_0 = o$.
The time $\tau$ is corresponding to a particular value of the exit time $S$, so $y_\ell \in D$
for $1 \le \ell < \tau$ and $y_{\tau} \notin D$. The parameter $C_1$ will be chosen in the course 
of the proof. See Figure \ref{fig:set_G} for a visual illustration of the sequence $y'_0, y'_1, \dots$ in the definition of $\mathcal{G}_{\zeta,C_1}$.

\begin{figure}[htb]%
\begin{center}
% --- Figure starts ---
\resizebox{0.7\textwidth}{!}{%
\begin{tikzpicture}
\draw (-5,-5) -- (-5,5) -- (5,5) -- (5,-5) -- (-5,-5);
\draw[thick,<-] (-5,-5.3) -- (-0.5,-5.3);
\draw[thick,->] (0.5,-5.3) -- (5,-5.3);
\node[anchor=center] at (0,-5.3) {$L$};
\filldraw[black] (0, 0) circle (1pt);
\node[anchor=west] at (0,0) {$o$};
\draw (-0.5,-0.5) -- (-0.5,0.5) -- (0.5,0.5) -- (0.5,-0.5) -- (-0.5,-0.5);
\filldraw[color=gray, draw=black] (-0.2,0.2) circle(1.5mm);
\draw (-1.5,-0.5) -- (-1.5,0.5) -- (-0.5,0.5) -- (-0.5,-0.5) -- (-1.5,-0.5);
\draw[thick,<->] (-0.5,-1.7) -- (0.5,-1.7);
\node[anchor=north] at (0,-1.7) {$N$};
\filldraw[color=gray, draw=black] (-1.2,0.2) circle(1.5mm);
\draw (-0.5,0.5) -- (-0.5,1.5) -- (0.5,1.5) -- (0.5,0.5) -- (-0.5,0.5);
\filldraw[color=gray, draw=black] (-0.2,1.2) circle(1.5mm);
\draw (-1.5,0.5) -- (-1.5,1.5) -- (-0.5,1.5) -- (-0.5,0.5) -- (-1.5,0.5);
\filldraw[color=gray, draw=black] (-1.2,1.2) circle(1.5mm);
\draw (-0.5,-1.5) -- (-0.5,-0.5) -- (0.5,-0.5) -- (0.5,-1.5) -- (-0.5,-1.5);
\filldraw[color=gray, draw=black] (-0.2,-0.8) circle(1.5mm);
\draw (-1.5,-1.5) -- (-1.5,-0.5) -- (-0.5,-0.5) -- (-0.5,-1.5) -- (-1.5,-1.5);
\filldraw[color=gray, draw=black] (-1.2,-0.8) circle(1.5mm);

% --- dashed blocks ---
% \draw[dashed] (-2.5,-0.5) -- (-2.5,0.5) -- (-1.5,0.5) -- (-1.5,-0.5) -- (-2.5,-0.5);
% \filldraw[color=gray!50, draw=black,dashed] (-2.2,0.2) circle(1.5mm);
% \draw[dashed] (-2.5,-1.5) -- (-2.5,-0.5) -- (-1.5,-0.5) -- (-1.5,-1.5) -- (-2.5,-1.5);
% \filldraw[color=gray!50, draw=black,dashed] (-2.2,-0.8) circle(1.5mm);
% \draw[dashed] (-2.5,-1.5) -- (-2.5,-0.5) -- (-1.5,-0.5) -- (-1.5,-1.5) -- (-2.5,-1.5);
% \filldraw[color=gray!50, draw=black,dashed] (-2.2,-0.8) circle(1.5mm);

% --- 4 blocks ---
\draw (0.5,1.5) -- (0.5,2.5) -- (1.5,2.5) -- (1.5,1.5) -- (0.5,1.5);
\filldraw[color=gray, draw=black] (0.8,2.2) circle(1.5mm);
\draw (0.5,0.5) -- (0.5,1.5) -- (1.5,1.5) -- (1.5,0.5) -- (0.5,0.5);
\filldraw[color=gray, draw=black] (0.8,1.2) circle(1.5mm);
\draw (0.5,-0.5) -- (0.5,0.5) -- (1.5,0.5) -- (1.5,-0.5) -- (0.5,-0.5);
\filldraw[color=gray, draw=black] (0.8,0.2) circle(1.5mm);
\draw (0.5,-1.5) -- (0.5,-0.5) -- (1.5,-0.5) -- (1.5,-1.5) -- (0.5,-1.5);
\filldraw[color=gray, draw=black] (0.8,-0.8) circle(1.5mm);

\draw (1.5,1.5) -- (1.5,2.5) -- (2.5,2.5) -- (2.5,1.5) -- (1.5,1.5);
\filldraw[color=gray, draw=black] (1.8,2.2) circle(1.5mm);
\draw (1.5,0.5) -- (1.5,1.5) -- (2.5,1.5) -- (2.5,0.5) -- (1.5,0.5);
\filldraw[color=gray, draw=black] (1.8,1.2) circle(1.5mm);
\draw (1.5,-0.5) -- (1.5,0.5) -- (2.5,0.5) -- (2.5,-0.5) -- (1.5,-0.5);
\filldraw[color=gray, draw=black] (1.8,0.2) circle(1.5mm);
\draw (1.5,-1.5) -- (1.5,-0.5) -- (2.5,-0.5) -- (2.5,-1.5) -- (1.5,-1.5);
\filldraw[color=gray, draw=black] (1.8,-0.8) circle(1.5mm);

% \draw (2.5,1.5) -- (2.5,2.5) -- (3.5,2.5) -- (3.5,1.5) -- (2.5,1.5);
% \draw (2.5,0.5) -- (2.5,1.5) -- (3.5,1.5) -- (3.5,0.5) -- (2.5,0.5);
\draw (2.5,-0.5) -- (2.5,0.5) -- (3.5,0.5) -- (3.5,-0.5) -- (2.5,-0.5);
\filldraw[color=gray, draw=black] (2.8,0.2) circle(1.5mm);
\draw (2.5,-1.5) -- (2.5,-0.5) -- (3.5,-0.5) -- (3.5,-1.5) -- (2.5,-1.5);
\filldraw[color=gray, draw=black] (2.8,-0.8) circle(1.5mm);
\draw (1.5,-2.5) -- (1.5,-1.5) -- (2.5,-1.5) -- (2.5,-2.5) -- (1.5,-2.5);
\filldraw[color=gray, draw=black] (2.8,-1.8) circle(1.5mm);
\draw (2.5,-2.5) -- (2.5,-1.5) -- (3.5,-1.5) -- (3.5,-2.5) -- (2.5,-2.5);
\filldraw[color=gray, draw=black] (2.8,-2.8) circle(1.5mm);
\draw (2.5,-3.5) -- (2.5,-2.5) -- (3.5,-2.5) -- (3.5,-3.5) -- (2.5,-3.5);
\filldraw[color=gray, draw=black] (1.8,-1.8) circle(1.5mm);
\draw (3.5,-2.5) -- (3.5,-1.5) -- (4.5,-1.5) -- (4.5,-2.5) -- (2.5,-2.5);
\filldraw[color=gray, draw=black] (3.8,-1.8) circle(1.5mm);
\draw (3.5,-3.5) -- (3.5,-2.5) -- (4.5,-2.5) -- (4.5,-3.5) -- (3.5,-3.5);
\filldraw[color=gray, draw=black] (3.8,-2.8) circle(1.5mm);

% --- y_1 and random walk ----
\filldraw[black] (1.2, 2.0) circle (1pt);
\node[anchor=north] at (1.2, 2.0) {$y'_1$};
\draw[thick] (0,0) .. controls (0.2,0.4) and (0.3,0.5) .. (0.5,0.6)
	  (0.5,0.6) .. controls (0.9,0.75) and (0.75,1.1) .. (0.8,1.2)
	  (0.8,1.2) .. controls (0.9,1.4) and (1.0,1.9) .. (1.2,2.0);

% --- y_2 and random walk ----
\filldraw[black] (2.2, -0.1) circle (1pt);
\node[anchor=east] at (2.2, -0.2) {$y'_2$};
\draw[thick] (1.2,2.0) .. controls (1.5,2.2) and (1.6,1.9) .. (1.7,1.8)
(1.7,1.8) .. controls (1.8,1.6) and (1.9,1.7) .. (2.0,1.5)
	  (2.0,1.5) .. controls (2.1,1.2) and (2.3,1.0) .. (2.2,0.8)
	  (2.2,0.8) .. controls (2.0,0.6) and (2.4,0.3) .. (2.2,-0.1);

% --- y_3 and random walk ----
\filldraw[black] (3.2, -2.1) circle (1pt);
\node[anchor=east] at (3.2, -2.2) {$y'_3$};
\draw[thick] (2.2,-0.1) .. controls (2.0,-0.3) and (2.1,-0.5) .. (2.2,-0.7)
(2.2,-0.7) .. controls (2.4,-0.9) and (2.3,-1.0) .. (2.5,-1.2)
(2.5,-1.2) .. controls (2.8,-1.4) and (3.0,-1.3) .. (3.2,-1.5)
(3.2,-1.5) .. controls (3.3,-1.7) and (3.1,-1.9) .. (3.2,-2.1);

% --- y_4 and random walk ----
\filldraw[black] (5.5, -3.0) circle (1pt);
\node[anchor=west] at (5.5, -3.0) {$y'_4 = y'_{\tau}$};
\draw[thick] (3.2,-2.1) .. controls (3.4,-2.4) and (3.2,-2.6) .. (3.3,-2.8)
(3.3,-2.8) .. controls (3.5,-3.0) and (3.7,-2.8) .. (4.0,-3.1)
(4.0,-3.1) .. controls (4.2,-3.2) and (4.5,-2.8) .. (4.8,-3.1)
(4.8,-3.1) .. controls (5.1,-3.3) and (5.3,-2.9) .. (5.5,-3.0);

\end{tikzpicture}
}
% --- Figure ends ---
\caption{This figure explains the properties of the set $\mathcal{G}_{\zeta,C_1}$ (not to scale). The shaded regions represent the balls of radius $N^\zeta$ in each copy of the torus.
None of the $y'_\ell$'s, for $\ell \ge 1$, is in a shaded region.}%
\label{fig:set_G}%
\end{center}
\end{figure}
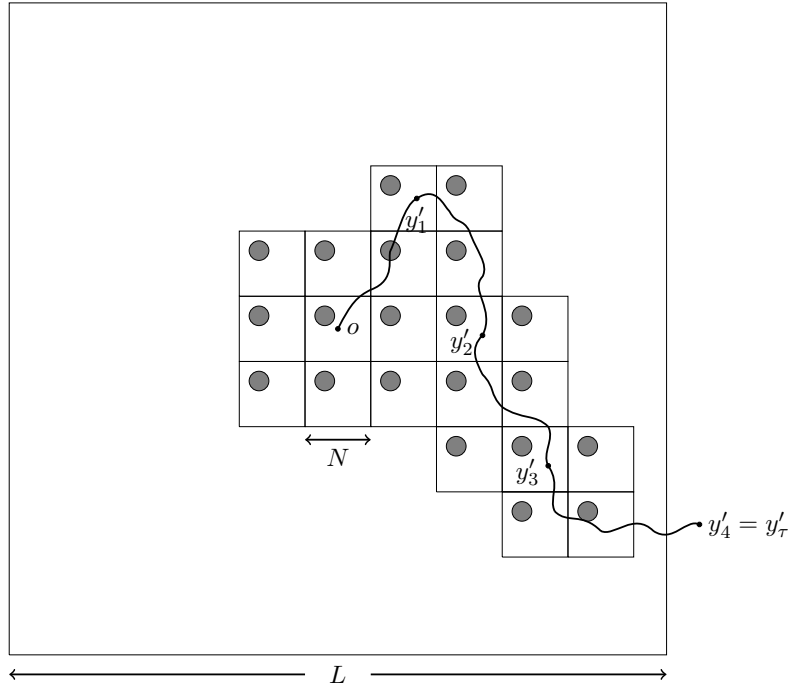

The next lemma shows that the restriction made on the time-parameter $\tau$ in the definition of $\mathcal{G}_{\zeta,C_1}$ holds with high probability for $S$.

\begin{lemma}
\label{lem:E1_prelim}
We have $\mathbf{P}_{o} \left[ \left\{\frac{S n}{N^d} < (\sqrt{\log\log n})^{-1} \right\} \cup \left\{\frac{S n}{N^d} > \log N \right\}\right] \to 0$ as $N \to \infty$.
\end{lemma}

\begin{proof}
By the definitions of $S$ and $T$, we first notice that 
\begin{equation*}
\begin{split}
&\P_{o}\left[S < (\sqrt{\log\log n})^{-1}\frac{N^d}{n} \right] \\
&\quad\leq \P_{o}\left[ T < (\sqrt{\log\log n})^{-1} N^d \right] \\
&\quad\leq \sum_{1\leq i\leq d}
	\left(\P_{o}\left[\max_{0\leq j \leq (\sqrt{\log\log n})^{-1}N^d} Y^{(i)}_{j}\geq L\right]
	+\P_{o}\left[\max_{0\leq j \leq (\sqrt{\log\log n})^{-1}N^d} -Y^{(i)}_{j}\geq L\right]\right),
\end{split}
\end{equation*}
where $Y^{(i)}$ denotes the $i$-th coordinate of the $d$-dimensional lazy random walk.

%Recall the Martingale maximal inequality \eqref{e:maximal_ineq}, 
%\begin{equation*}
%\P\left[\max_{0\leq j\leq t} Y^{(i)}_j \geq r\sigma\sqrt{t} \right]
%\leq e^{-r^2/2}\exp\left\{O\left(\frac{r^3}{\sqrt{t}}\right)\right\},
%\end{equation*}
%where, in our case, $\sigma$ is the variance of a single coordinate for the lazy random walk, 
%$\sigma^2 = 1/2d$.

We are going to use \eqref{e:maximal_ineq}. 
Setting $t= (\sqrt{\log\log n})^{-1} N^d$ and $r\sigma\sqrt{t} = L$, we can evaluate each term 
(similarly for the event with $-Y^{(i)}_j$) in the sum
\begin{equation}
\label{e:maximal_exp}
\begin{split}
 &\P_{o}\left[\max_{0\leq j \leq (\sqrt{\log\log n})^{-1}N^d} Y^{(i)}_{j}\geq L\right] \\
 &\quad\leq \exp\left\{-\frac{1}{2}\frac{L^2}{\sigma^2(\sqrt{\log\log n})^{-1} N^d}
	 + O\left(\frac{L^3}{\sigma^3(\sqrt{\log\log n})^{-2}N^{2d}}\right)\right\}.
\end{split}
\end{equation}
Recall that $L^2 \sim A N^d$ and $\sigma^2 = 1/2d$. For the main term in the exponential in \eqref{e:maximal_exp}, we have the upper bound
\begin{equation*}
\begin{split}
\exp\left( -(1+o(1)) \frac{1}{2}\frac{2d\cdot A N^d}{ (\sqrt{\log\log n})^{-1} N^d}\right)
&= \exp\left(-(1+o(1)) Ad\sqrt{\log\log n}\right)\\
&\to 0, \quad\text{as $N\to\infty$}.
\end{split}
\end{equation*}
The big $O$ term in the exponential in \eqref{e:maximal_exp} produces an error term because
\begin{equation*}
\begin{split}
 \exp\left\{O\left(\frac{(A N^d)^{3/2}}{\sigma^3(\sqrt{\log\log n})^{-2}N^{2d}}\right)\right\}
 &= \exp \left\{O\left(N^{-d/2}(\log\log n)\right)\right\}\\
 &= 1+o(1), \quad\text{as $N\to\infty$.}
\end{split}
\end{equation*}

%Now we use the Central Limit Theorem (CLT) and induction to estimate the probability $\P \left[ \frac{S n}{N^d} > \log N \right]$. 
%By the CLT, if we have a lazy random walk starting from the origin $o$ at time $0$, the walk exits $[-L,L)^d$ in $N^d$ steps with 
%at least positive probability $c_0$ independent of $N$, for sufficient large $N$.
%
%\begin{align*}
%\P_{o}\left[S > \frac{N^d}{n}\right]
%\leq \P_{o}\left[Y_{N^d}\in [-L,L)^d\right]
%= 1 - \P_{o}\left[Y_{N^d}\notin [-L,L)^d\right]
%\leq 1- c_{0}.
%\end{align*}
%
%Assume that
%\begin{equation*}
  %\P_{o}\left[S > k\frac{N^d}{n}\right]
  %\leq (1-c_0)(1-c'_0)^{k-1}
  %\leq e^{-c'_0 k},  
%\end{equation*}
%where $c_0$ and $c'_0$ comes from the CLT with $c'_0\leq c_0$; see below 
%for details.
%
%By the Markov property,
%\begin{equation*}
%\begin{split}
   %\P_{o}\left[S > (k+1)\frac{N^d}{n}\right]
%&= \P_{o}\left[S > k\frac{N^d}{n}\right]
   %\P_{o}\left[S > (k+1)\frac{N^d}{n}\middle|S > k\frac{N^d}{n}\right]\\
%&\leq \P_{o}\left[S > k\frac{N^d}{n}\right]
	  %\max_{z\in[-L,L)^d}\P_{z}\left[Y_{N^d}\in [-L,L)^d\right]\\
%&\leq \P_{o}\left[S > k\frac{N^d}{n}\right]
	  %\max_{z\in[-L,L)^d}\P_{z}\left[Y_{N^d}\in z+[-2L,2L)^d\right] \\
%&\leq \P_{o}\left[S > k\frac{N^d}{n}\right]\P_{o}\left[Y_{N^d}\in [-2L,2L)^d\right]\\
%&\leq (1-c_0)(1-c'_0)^{k-1}(1-c'_0)
%=(1-c_0)(1-c'_0)^{k}
%\leq e^{-c'_0(k+1)},
%\end{split}
%\end{equation*}
%where $c_0$ and $c'_0$ comes from the CLT with $c'_0\leq c_0$.

Coming to the second event $\left\{\frac{S n}{N^d} > \log N \right\}$, 
observe that the Central Limit Theorem applied to $(Y_{k n \lfloor N^d/n \rfloor) + t} - Y_{k n \lfloor N^d/n \rfloor})_{t \ge 0}$
implies that 
\[ \P_o \Big[ S > (k+1) \Big\lfloor \frac{N^d}{n} \Big\rfloor \,\Big|\, S > k \Big\lfloor \frac{N^d}{n} \Big\rfloor \Big]
   \le \max_{z \in (-L,L)^d}  (1 - \P_z [ Y_{n \lfloor N^d/n \rfloor} \not\in (-2L,2L)^d ])
	\le 1-c \]
for some $c > 0$. Hence we have 
$\P_{o}\left[S > k\frac{N^d}{n}\right]\leq e^{-c k}$ for all $k \ge 0$.

Applying this with $k=\log N$, 
\begin{equation*}
  \P_{o}\left[S > \log N\frac{N^d}{n}\right] 
  \leq e^{-c \log N} \to 0   
\end{equation*}
as required.

\end{proof}

%For this, we only want to consider stretches with ``well-behaved'' starting and end-points, which motivates the definition of the set $\mathcal{G}_{\zeta,C_1}$ below. 

% commented out good-stretches with $\zeta_2$ and lemma 5 ($E5$ bound) 
%\begin{equation*}
%\mathcal{G}_{\zeta}
%= \left\{ \left( \tau, (y_\ell, \mathbf{y}_\ell)_{\ell=1}^\tau \right) : 
%   \parbox{6.5cm}{$(\log N)^{-1} \frac{N^d}{n} \le \tau \le \log N \frac{N^d}{n}$, 
%	 $y_\ell \in D$, $\bfy_\ell \in \T_N \setminus B(\mathbf{x_0}, N^{\zeta})$ for
%	 $1 \le \ell < \tau$, 
%	 $N^{\zeta_2}\leq |(y_{\ell}N+\bfy_{\ell})-(y_{\ell-1}N+\bfy_{\ell-1})|
%	 \leq f(n)n^{\frac{1}{2}}$,
%	 $y_\tau \in D^c$} \right\}. 
%\end{equation*}

The starting point for the proof of Theorem \ref{thm:main} is the following proposition that decomposes 
the probability we are interested in into terms involving single stretches of 
duration $n$.

\begin{proposition}
\label{prop:good_stretches}
For a sufficiently large value of $C_1$, we have that 
\begin{equation}
\label{e:high-prob}
\P_o \left[ \left(S, (Y_{\ell n}, \varphi(Y_{\ell n})_{\ell=1}^S \right) \not\in \mathcal{G}_{\zeta,C_1} \right]
   = o(1) \quad \text{as $N \to \infty$}.
\end{equation}
Furthermore, 
\begin{equation}
\label{e:good_stretches}
\begin{split}
 &\mathbf{P}_{o} \left[ Y_t \not\in \varphi^{-1}(\mathbf{K}),\, 0 \le t < T \right] \\
 &\quad = \sum_{( \tau, (y_\ell, \mathbf{y}_\ell)_{\ell=1}^\tau ) \in \mathcal{G}_{\zeta,C_1}}
   \prod_{\ell=1}^\tau 
%	 \mathbf{P}_{\bfy_{\ell-1}+y_{\ell-1}N} 
%	 \left[ Y_n = \bfy_{\ell}+y_{\ell}N,\, \text{$Y_t \not\in \varphi^{-1}(\mathbf{K})$ for 
	 \mathbf{P}_{y'_{\ell-1}} 
	 \left[ Y_n = y'_\ell,\, \text{$Y_t \not\in \varphi^{-1}(\mathbf{K})$ for 
	 $0 \le t < n$} \right]
	 + o(1), 
\end{split}
\end{equation}
where $o(1) \to 0$ as $N \to \infty$.
\end{proposition}

%...add statements of other propositions here

Central to the proof of Theorem \ref{thm:main} is the following proposition, that estimates the 
probability of hitting a copy of $\bfK$ during a ``good stretch'' where the displacement 
$|y'_\ell - y'_{\ell-1}|$ is almost of order $\sqrt{n}$. This will
not hold for all stretches with high probability, but the fraction of stretches for which 
it fails will vanish asymptotically.

\begin{proposition}
\label{prop:capacity_ub}
There exists a sufficiently large value of $C_1$ such that the following holds. 
Let $( \tau, (y_\ell, \mathbf{y}_\ell)_{\ell=1}^\tau ) \in \mathcal{G}_{\zeta,C_1}$.
Then for all $2 \le \ell \le \tau$ such that 
$|y'_\ell - y'_{\ell-1}| \le 10\sqrt{n}\log\log n$ we have
\begin{equation}
\label{e:capacity_ub}
\begin{split}
 &\mathbf{P}_{y'_{\ell-1}} 
	 \left[ Y_n = y'_\ell,\, \text{$Y_t \not\in \varphi^{-1}(\mathbf{K})$ for 
	 $0 \le t < n$} \right] \\
 &\qquad = \mathbf{P}_{y'_{\ell-1}} 
	 \left[ Y_n = y'_\ell \right] \, 
	 \left( 1 - \frac{1}{2} \, \frac{\cpty (\mathbf{K}) \, n}{N^d} \, (1 + o(1)) \right).
\end{split}
\end{equation}
\end{proposition}

In addition to the above proposition (that we prove in Section \ref{ssec:capacity_ub}), 
we will need a weaker version for the remaining ``bad stretches'' that have less restriction 
on the distance $|y'_\ell - y'_{\ell-1}|$. This will be needed to estimate error terms arising 
from the ``bad stretches'', and it will also be useful to demonstrate some of our proof ideas 
for other error terms arising later in the paper. It will be proved in 
Section \ref{ssec:weaker_capacity_ub}.

%In addition to this proposition (that we prove in Section \ref{ssec:capacity_ub}), 
%however, we will also need a weaker version of it under less 
%restriction on the distance $|y'_\ell - y'_{\ell-1}|$. 
%We will prove this weaker version in Section \ref{ssec:weaker_capacity_ub}. 
%This version will be needed to handle some error terms, and it will also be useful to demonstrate 
%some of our proof ideas first in this simpler setting.  

\begin{proposition}
\label{prop:weaker_capacity_ub}
Let $( \tau, (y_\ell, \mathbf{y}_\ell)_{\ell=1}^\tau ) \in \mathcal{G}_{\zeta,C_1}$.
For all $2 \le \ell \le \tau$ we have
\begin{equation}
\label{e:weaker_capacity-ub}
 \P_{y'_{\ell-1}} 
	 \left[ Y_n = y'_\ell,\, \text{$Y_t \not\in \varphi^{-1}(\mathbf{K})$ for all
	 $0 \le t < n$} \right] 
 = \P_{y'_{\ell-1}} 
	 \left[ Y_n = y'_\ell \right] \, 
	 \left( 1 - O\left(\frac{n}{N^d}\right) \right).
\end{equation}
and for the first stretch we have
\begin{equation}
\label{e:weaker_capacity-ub-o}
 \P_{o} 
	 \left[ Y_n = y'_1,\, \text{$Y_t \not\in \varphi^{-1}(\mathbf{K})$ for all
	 $0 \le t < n$} \right] 
 = \P_{o} 
	 \left[ Y_n = y'_1 \right] \, 
	 \left( 1 - o(1) \right).
\end{equation}
Here the $-O(n/N^d)$ and $-o(1)$ error terms are negative.
\end{proposition}

Our final proposition is needed to estimate the number of stretches that are ``bad''.

\begin{proposition}
\label{prop:number_l}
We have
\begin{equation}
\begin{split}
 &\P_o \left[ \# \left\{ 1 \le \ell \le \frac{N^d}{n} C_1 \log N : 
     |Y_{n \ell} - Y_{n (\ell-1)}| > 10\sqrt{n}\log\log n \right\}
     \ge \frac{N^d}{n} \frac{1}{\log \log n} \right] \\
 &\qquad \to 0,
\end{split}
\end{equation}
as $N \to \infty$.
\end{proposition}

\section{Proof of the main theorem assuming the key propositions} % \ref{prop:good_stretches}--\ref{prop:number_l}}
\label{sec:proof-of-main}

This section is the proof of Theorem \ref{thm:main}.

\begin{proof}[Proof of Theorem \ref{thm:main} assuming Propositions \ref{prop:good_stretches}--\ref{prop:number_l}] \ \\ 
First, given any $(\tau,(y_\ell, \bfy_\ell)_{\ell=1}^\tau) \in \mathcal{G}_{\zeta,C_1}$, we define
\begin{equation}
\begin{split}
  \cL
	&= \{2 \le \ell \le \tau : |y'_\ell-y'_{\ell-1}|\le 10\sqrt{n}\log\log n\} \\
	\cL'
	&= \{2 \le \ell \le \tau : |y'_\ell-y'_{\ell-1}|> 10\sqrt{n}\log\log n\}.
\label{e:cL-decomp}
\end{split}
\end{equation} 
Thus we have
\[ \{ 1, \dots, \tau \}
   = \{ 1 \} \cup \cL \cup \cL'. \]
We further denote by 
\begin{equation*}
    \mathcal{G}'_{\zeta,C_1}
    = \left\{ (\tau,(y_\ell, \bfy_\ell)_{\ell=1}^\tau) \in \mathcal{G}_{\zeta,C_1} : |\cL'| \le \frac{N^d}{n}\frac{1}{\log\log n} \right\}.
\end{equation*}

We have by Proposition \ref{prop:good_stretches} that 
\begin{equation}
\begin{split}
 &\mathbf{P}_{o} \left[ Y_t \not\in \varphi^{-1}(\mathbf{K}),\, 0 \le t < T \right] \\
 & = o(1) + \sum_{( \tau, (y_\ell, \mathbf{y}_\ell)_{\ell=1}^\tau ) \in \mathcal{G}_{\zeta,C_1}}
          \prod_{\ell=1}^\tau \mathbf{P}_{y'_{\ell-1}} [ Y_n = y'_\ell,\, \text{$Y_t \not\in \varphi^{-1}(\mathbf{K})$ for 
	 $0 \le t < n$} ].
\end{split}
\label{e:decomp-initial}
\end{equation}
By Proposition \ref{prop:number_l}, we can replace the summation over elements of $\mathcal{G}_{\zeta,C_1}$ by summation 
over just elements of $\mathcal{G}'_{\zeta,C_1}$, at the cost of a $o(1)$ term.
That is,
\begin{equation}
\begin{split}
 &\mathbf{P}_{o} \left[ Y_t \not\in \varphi^{-1}(\mathbf{K}),\, 0 \le t < T \right] \\
 & = o(1) + \sum_{( \tau, (y_\ell, \mathbf{y}_\ell)_{\ell=1}^\tau ) \in \mathcal{G}'_{\zeta,C_1}}
          \prod_{\ell=1}^\tau \mathbf{P}_{y'_{\ell-1}} [ Y_n = y'_\ell,\, \text{$Y_t \not\in \varphi^{-1}(\mathbf{K})$ for 
	 $0 \le t < n$} ].
\end{split}
\label{e:decomp-initial2}
\end{equation}
Applying Proposition \ref{prop:capacity_ub} for the factors $\ell \in \cL$ and Proposition \ref{prop:weaker_capacity_ub}
for the factors $\ell \in \{ 1 \} \cup \cL'$ we get
\begin{equation}
\label{e:decomp-main}
\begin{split}
 &\mathbf{P}_{o} \left[ Y_t \not\in \varphi^{-1}(\mathbf{K}),\, 0 \le t < T \right] \\
 & = o(1) + \sum_{( \tau, (y_\ell, \mathbf{y}_\ell)_{\ell=1}^\tau ) \in \mathcal{G}'_{\zeta,C_1}}
          \prod_{\ell=1}^\tau \mathbf{P}_{y'_{\ell-1}} [ Y_n = y'_\ell ] \\
 &\qquad  \times (1-o(1)) \, \prod_{\ell\in \cL}
          \left(1-\frac{1}{2}\cpty(\bfK)\frac{n}{N^d}(1+o(1))\right)
          \,\prod_{\ell\in \cL'} \left(1-O\left(\frac{n}{N^d}\right)\right).
\end{split}
\end{equation}
%Here the $-o(1)$ and $-O(n/N^d)$ terms are negative.

%By Proposition \ref{prop:number_l}, we have that we can replace the summation over elements of $\mathcal{G}_{\zeta,C_1}$ by summation over just elements of $\mathcal{G'}_{\zeta,C_1}$, at the cost of a $o(1)$ term.
%That is,
%\begin{equation}
%\label{e:decomp-main2}
%\begin{split}
% &\mathbf{P}_{o} \left[ Y_t \not\in \varphi^{-1}(\mathbf{K}),\, 0 \le t < T \right] \\
% & = o(1) + \sum_{( \tau, (y_\ell, \mathbf{y}_\ell)_{\ell=1}^\tau ) \in \mathcal{G'}_{\zeta,C_1}}
%          \prod_{\ell=1}^\tau \mathbf{P}_{y'_{\ell-1}} [ Y_n = y'_\ell ] \\
% &\qquad  \times (1-o(1)) \, \prod_{\substack{2\le\ell\le\tau\\ \ell\in \cL}}
%          \left(1-\frac{1}{2}\cpty(\bfK)\frac{n}{N^d}(1+o(1))\right)
%          \,\prod_{\substack{2\le\ell\le\tau\\ \ell\in \cL^c}}\left(1-O\left(\frac{n}{N^d}\right)\right).
%\end{split}
%\end{equation}
Note that since the summation is over elements of $\mathcal{G}'_{\zeta,C_1}$ only, we have 
\begin{equation}
\label{e:L^c_bound}
|\cL'| \le \frac{N^d}{n}\frac{1}{\log\log n}.
\end{equation}
By \eqref{e:L^c_bound}, we can lower bound the last product in \eqref{e:decomp-main} by
\begin{equation*}
\exp\left(-O\left(\frac{n}{N^d}\right)\frac{N^d}{n}\frac{1}{\log\log n}\right)
= e^{o(1)} = (1+o(1)).
\end{equation*}
Since the product is also at most $1$, it equals $1 + o(1)$.

Also, due to \eqref{e:L^c_bound}, we have 
\begin{equation*}
 \tau-1-\frac{N^d}{n}\frac{1}{\log\log n}
 \le |\cL| \le\tau.
\end{equation*}
Since $\tau \ge \frac{N^d}{n} (\sqrt{\log \log n})^{-1}$, we have $|\cL|=(1+o(1))\tau$.
This implies that the penultimate product in \eqref{e:decomp-main} equals
\begin{equation}
 \left(1-\frac{1}{2}\cpty(\bfK)\frac{n}{N^d}(1+o(1))\right)^{(1+o(1))\tau}
 = \exp \left( - \frac{1}{2} \cpty(\bfK) \frac{n}{N^d} \tau (1+o(1)) \right). 
\end{equation}

Recall that $S=\inf\{\ell\geq 0: Y_{n\ell}\notin (-L,L)^d\}$.
By summing over $(y_\ell, \mathbf{y}_\ell)_{\ell=1}^\tau$ and appealing to \eqref{e:high-prob}, 
we get that \eqref{e:decomp-main} equals
\begin{equation}
 o(1) + \sideset{}{'}\sum_{\tau} \E \left[ \mathbf{1}_{S = \tau} 
   \exp \left( - \frac{1}{2} \cpty(\bfK) \frac{n}{N^d} \tau (1+o(1)) \right) \right], 
\end{equation}
where the primed summation denotes restriction to 
$\frac{N^d}{n} (\sqrt{\log \log n})^{-1} \le \tau \le (\log N) \frac{N^d}{n}$.
Since due to Lemma \ref{lem:E1_prelim}, $S$ satisfies the bounds on $\tau$ with probability going to $1$, the latter 
expression equals
\begin{equation}
 o(1) + \E \left[ e^{-\frac{1}{2} \cpty(\bfK) \frac{n}{N^d} S} \right]. 
\end{equation}

Let $\Gamma_n$ denote the covariance matrix for $Y_n$, so that 
$\Gamma_n=\frac{n}{2d}I$.
Let $Z_{1}=\sqrt{\frac{2d}{n}}\,Y_{n}$, 
with the covariance matrix $\Gamma_{Z} = I$.
Let $Z_\ell=\sqrt{\frac{2d}{n}}\,Y_{n\ell}$ for $\ell\ge 0$.

Since $L^2\sim A\,N^d$, the event $\{Y_{n\ell} \notin (-L,L)^d\}$ is 
the same as $\{Y_{n\ell} \notin (-(1+o(1))\sqrt{A}N^{d/2},(1+o(1))\sqrt{A}N^{d/2})^d\}$.
Converting to events in terms of $Z$ we have
\begin{equation*}
Z_{\ell} 
\notin \left(-\sqrt{2dA(1+o(1))}(N^d/n)^{1/2},\sqrt{2dA(1+o(1))}(N^d/n)^{1/2}\right)^d.
\end{equation*}
Now we can write $S$ as 
\begin{equation*}
S=\inf\{\ell\ge 0: 
Z_{\ell}\notin(-\sqrt{2dA(1+o(1))}(N^d/n)^{1/2},\sqrt{2dA(1+o(1))}(N^d/n)^{1/2})^d\}.
\end{equation*}

Let $\sigma_1 = \inf\{t>0:B_t\notin(-1,1)^d\}$ be the exit time of Brownian motion
from $(-1,1)^d$.
By Donsker's Theorem \cite[Theorem 8.1.5]{Durrett2019}
%and a result of Durrett [Example 8.1.12, \cite{Durrett2019}],
we have 
\begin{equation*}
\P\left[S\le 2dA(1+o(1))\frac{N^d}{n}t\right]
\to \P[\sigma_1\le t].
\end{equation*}
Then we have that $\frac{n}{N^d} S$ converges in distribution to
$c \sigma_1$, with $c = 2dA$. This completes the proof.
\end{proof}

\section{Proofs of the key propositions} % \ref{prop:good_stretches}--\ref{prop:number_l}}
\label{sec:proof of key propositions}

\subsection{Proof of Proposition \ref{prop:weaker_capacity_ub}}
\label{ssec:weaker_capacity_ub}

In the proof of the proposition we will need the following lemma
that bounds the probability of hitting some copy of $\bfK$ in terms 
of the Green's function of the random walk. Recall that the Green's 
function is defined by 
\[ G(x',y')
   = \sum_{t=0}^\infty p_t(x',y'), \]
and in all $d \ge 3$ satisfies the bound \cite{LawlerLimic2010}
\[ G(x',y') \le \frac{C_G}{|y'-x'|^{d-2}} \]
for a constant $C_G = C_G(d)$.
For part (ii) of the lemma recall that $\bfK \cap \varphi(B_{g(N)}(o)) = \es$.
We also define $\mathrm{diam}(\bfK)$ as the maximum Euclidean distance between two points in $\bfK$.

\begin{lemma}
\label{lem:Green-bnd}
Let $d \ge 3$. Assume that $N$ is large enough so that $N^\zeta \ge \mathrm{diam}(\bfK)$. \\
(i) If $y' \in \Z^d$ satisfies $\varphi(y') \not\in B(\bfx_0,N^{\zeta})$, then
for all sufficiently small $\eps > 0$ we have
\begin{equation}
\label{e:Green-bnd-short}
  \sum_{t=0}^{N^{2+6\eps}} \sum_{x \in \Z^d} \sum_{x' \in \bfK+xN}
  p_t(y',x')
	%\P_{y'} [ Y_t \in \varphi^{-1}(\bfK) ] 
	\le \frac{C}{N^{\zeta(d-2)}}.
\end{equation}
(ii) If $g(N) \le N^\zeta$, then for all sufficiently small $\eps > 0$ we have
\begin{equation}
\label{e:Green-bnd-o}
  \sum_{t=0}^{N^{2+6\eps}} \sum_{x \in \Z^d} \sum_{x' \in \bfK+xN}
  p_t(o,x')
	%\P_{o} [ Y_t \in \varphi^{-1}(\bfK) ] 
	\le \frac{C}{g(N)^{(d-2)}}.
\end{equation}
(iii) If $y' \in \Z^d$ satisfies $\varphi(y') \not\in B(\bfx_0,N^{\zeta})$, then
for all sufficiently small $\eps > 0$ we have
\begin{equation}
\label{e:Green-bnd}
  \sum_{x \in \Z^d} \sum_{\substack{x' \in \bfK+xN \\ |x'-y'| \le n^{\frac{1}{2}+\eps}}}
    G(y',x') 
	\le \frac{C}{N^{d-\delta-2\delta\eps}}.
\end{equation}
\end{lemma}

\begin{proof} (i) We split the sum according to whether $|y'-x'| > N^{1+4\eps}$ or $\le N^{1+4\eps}$. 
In the first case we use \eqref{e:gaussian_ub_lb} and write $r = \lfloor |x'-y'| \rfloor$ to get
\begin{equation*}
\begin{split}
    \sum_{t=0}^{N^{2+6\eps}}\sum_{\substack{x' \in \varphi^{-1}(\bfK) \\|x'-y'|> N^{1+4\eps}}} p_t(y', x')
    &\le \sum_{t=0}^{N^{2+6\eps}}\sum_{r=\lfloor N^{1+4\eps} \rfloor}^{\infty} C\,r^{d-1} \exp\left(-\frac{c\, r^2}{N^{2+6\eps}}\right)\\
    &\le N^{2+6\eps}\sum_{r=\lfloor N^{1+4\eps}\rfloor}^{\infty} C\, r^{d-1} \exp\left(-\frac{c \, r^2}{N^{2+6\eps}}\right)\\
    &\le N^{O(1)}\exp(-cN^{2\eps{}}).
\end{split}
\end{equation*}

%Hence, the sum of probability $p_t$ over $|Y_t - y'| > N^{1+4\eps}$ for some $0 \le t \le N^{2+6 \eps}$ is stretched-exponentially small in $N$. 
For the remaining terms, we have the upper bound
\begin{equation*}
    \sum_{t=0}^{N^{2+6\eps}}\sum_{\substack{x' \in \varphi^{-1}(\bfK) \\|x'-y'|\le N^{1+4\eps}}} p_t(y', x')
    \le \sum_{\substack{x' \in \varphi^{-1}(\bfK) \\|x'-y'|\leq N^{1+4\eps}}} G(y',x').
%  \leq \sum_{\bfx\in\bfK} \sum_{x : |xN+\bfx-y'|\leq N^{1+4\eps}}
%	 \frac{C_G}{|y'-x'|^{d-2}}.
\end{equation*}

Let $Q(k N)$ be the cube with radius $kN$ centred at $o$ and then $y' + \left( Q(k N)\backslash Q((k-1)N) \right)$
are disjoint annuli for $k = 1, 2, \dots$.
We decompose the sum over $x'$ according to which annulus $x'$ falls into. For $k \ge 2$ we have
\begin{align*}
\sum_{\substack{x'\in\varphi^{-1}(\bfK)
	\\x'-y'\in Q(k N)\backslash Q\left((k -1)N\right)}} 
	\frac{C_G}{|y'-x'|^{d-2}}
\leq |\bfK| C k^{d-1} C_G (N k)^{2-d}
\leq |\bfK| C k N^{2-d},
\end{align*}
where $C_G$ is the Green's function constant. The contribution from any copy of $\bfK$ in 
$y' + Q(N)$ will be of order $N^{2-d}$ if its distance from $y'$ is at least $N/3$, say.
Note that there is at most one copy of $\bfK$ within distance $N/3$ of $y'$, 
which may have a distance as small as $N^\zeta$.

%Recall that $n=N^{\delta}$, then 
We have to sum over the following values of $k$:
\begin{align*}
k=1,\dots, \frac{N^{1+4\eps}}{N}
= N^{4\eps}.
\end{align*}
Since $x'\in\varphi^{-1}(\bfK)$ and $y'\notin \varphi^{-1}\left(B(\bfx_0, N^{\zeta})\right)$
for $\bfx_0\in \bfK$, 
the distance between $x'$ and $y'$ is at least $N^{\zeta}$.
Therefore, we get the upper bound as follows:
\begin{align*}
\sum_{\substack{x'\in\varphi^{-1}(\bfK)\\|x'-y'|\leq N^{1+4\eps}}}G(y',x')
&\leq |\bfK|N^{\zeta(2-d)} + \sum_{k=1}^{N^{4 \eps}} |\bfK|CkN^{2-d}\\
&\leq |\bfK|N^{\zeta(2-d)} + C |\bfK|N^{2-d}\times N^{8\eps} 
\leq C|\bfK|N^{\zeta(2-d)}.
\end{align*}
Here the last inequality follows from the choice of $\zeta$, \eqref{e:choice-zeta}, for 
sufficiently small $\eps > 0$.

(ii) The proof is essentially the same, except for the contribution of the 
''nearest'' copy of $\bfK$, which is now $C |\bfK| g(N)^{2-d}$.

(iii) The proof is very similar to that in part (i). 
Recall that $n=\lfloor N^{\delta} \rfloor$. This time we need to sum over
$k = 1, \dots, n^{\frac{1}{2}+\eps} / N$, which results in the bound
\[ C |\bfK| N^{-\zeta(d-2)} + C|\bfK| N^{2-d} \times N^{\delta + 2 \delta \eps - 2}
  = C|\bfK| \left[ N^{-\zeta(d-2)} + N^{\delta - d + 2 \delta \eps} \right]. \]
Here, for $\eps > 0$ small enough, the second term dominates due to the choice of $\zeta$; see \eqref{e:choice-zeta}.
\end{proof}

\begin{proof}[Proof of Proposition \ref{prop:weaker_capacity_ub}]
Since 
\begin{equation*}
\begin{split}
 &\P_{y'_{\ell-1}} 
	 \left[ Y_n = y'_\ell,\, \text{$Y_t \not\in \varphi^{-1}(\mathbf{K})$ for 
	 $0 \le t < n$} \right] \\
 &\qquad = \P_{y'_{\ell-1}} \left[ Y_n = y'_\ell \right]
   - \P_{y'_{\ell-1}} 
	 \left[ Y_n = y'_\ell,\, \text{$Y_t \in \varphi^{-1}(\mathbf{K})$ for
	 some $0 \le t < n$} \right], 
\end{split}
\end{equation*}
we need to show that 
\begin{equation*}
 \P_{y'_{\ell-1}} 
	 \left[ Y_n = y'_\ell,\, \text{$Y_t \in \varphi^{-1}(\mathbf{K})$ for some
	 $0 \le t < n$} \right]
 = O\left(\frac{n}{N^d}\right)\P_{y'_{\ell-1}} \left[ Y_n = y'_\ell \right].
\end{equation*}

Define $A(x)= \left\{Y_n=y'_{\ell},
\text{ $Y_t\in x N + \bfK$ for some $0\leq t < n$} \right\}$, 
so that 
\begin{equation}
\label{e:sum-A(x)-ub}
 \mathbf{P}_{y'_{\ell-1}} 
	 \left[ Y_n = y'_\ell,\, \text{$Y_t \in \varphi^{-1}(\mathbf{K})$ for some
	 $0 \le t < n$} \right] 
 \le \sum_{x \in \Z^d} \P_{y'_{\ell-1}} [ A(x) ]. 
\end{equation}
%\begin{equation}
%\begin{split}
% \mathbf{P}_{y'_{\ell-1}} 
%	 \left[ Y_n = y'_\ell,\, \text{$Y_t \in \varphi^{-1}(\mathbf{K})$ for 
%	 $0 \le t < n$} \right]
% &= \P_{y'_{\ell-1}} [\cup_{x\in\Z^d} A(x)] \\
% &\le \sum_{x \in \Z^d} \P_{y'_{\ell-1}} [ A(x) ].
%\end{split} 
%\end{equation}
We have
\begin{equation}
\label{e:A(x)-ub}
\begin{split}
\P_{y'_{\ell-1}}[A(x)]
	\leq \sum_{n_1+n_2=n}\sum_{x'\in \bfK+xN}
		p_{n_1}(y'_{\ell-1},x') p_{n_2}(x',y'_{\ell}).
\end{split}
\end{equation}

We bound this by splitting up the sum into different contributions. Let $\eps > 0$ 
that will be chosen sufficiently small in the course of the proof.

\emph{Case 1. $n_1, n_2\geq N^{2+6\eps}$ and $|y'_{\ell-1}-x'|\leq n^{\frac{1}{2}+\eps}_1$,
$|x'-y'_{\ell}|\leq n^{\frac{1}{2}+\eps}_2$.} By the LCLT  
we have that 
\begin{equation}
\label{e:LCLT-Case1}
\begin{split}
 p_{n_1}(y'_{\ell-1},x')
	&\leq Cp_{n_1}(y'_{\ell-1},u') \quad
	\text{for any $u' \in \T_N+xN$}, \\
  p_{n_2}(x',y'_{\ell})
	&\leq Cp_{n_2}(u',y'_{\ell}) \quad
	\text{for any $u' \in \T_N+xN$}. 
\end{split}
\end{equation} 
For this note that we have
\begin{equation*}
\begin{split}
\left|\frac{d|y'_{\ell-1}-x'|^2}{n_1} - \frac{d|y'_{\ell-1}-u'|^2}{n_1}\right|
&\leq \frac{d|x'-u'|^2}{n_1} 
+ \frac{2d|\langle x'-u',y'_{\ell-1}-x'\rangle|}{n_1} \\
&\leq C\frac{N^2}{n_1} + \frac{CN\cdot n_1^{\frac{1}{2}+\eps}}{n_1},
\end{split}
\end{equation*}
%\begin{equation*}
% \left|\frac{d|y'_{\ell-1}-x'|^2}{n_1} - \frac{d|y'_{\ell-1}-u'|^2}{n_1}\right|
% = \left|\frac{d|y'_{\ell-1}-x'|^2}{n_1} 
%	- \frac{d|y'_{\ell-1}-x'+x'-u'|^2}{n_1}\right|
%\end{equation*}
where the first term tends to $0$ and the rest equals
\begin{equation*}
\begin{split}
CN n_1^{-\frac{1}{2}+\eps}
\leq CN\cdot N^{(2+6\eps)(-\frac{1}{2}+\eps)} 
= CN^{-\eps+6\eps^2}
\to 0, \quad\text{as $N\to \infty$.}
\end{split}
\end{equation*}
Here we choose $\eps$ so that $-\eps+6\eps^2 <0$.
A similar observation shows the estimate for $p_{n_2}(x',y'_\ell)$.

The way we are going to use \eqref{e:LCLT-Case1} is to replace 
the summation over $x'$ by a summation over all $u' \in \T_N + xN$, and at the same time inserting a factor $|\bfK|/N^d$. Hence the contribution of the values of $n_1, n_2$ and $x$ in Case 1 to the right-hand side of
\eqref{e:sum-A(x)-ub} is at most
\begin{equation*}
\begin{split}
 \frac{C|\bfK|}{N^d} \sum_{n_1 + n_2 = n} \sum_{u' \in \Z^d} p_{n_1}(y'_{\ell-1},u') p_{n_2}(u', y'_\ell)
 &= \frac{C|\bfK|}{N^d} \sum_{n_1 + n_2 = n} p_{n}(y'_{\ell-1}, y'_{\ell}) \\
 &\le \frac{C|\bfK| n}{N^d} p_{n}(y'_{\ell-1}, y'_\ell).
\end{split}
\end{equation*}

%where $1/N^d$ comes from that $u'$ is uniformly chosen from $\T_N+xN$.
This completes the bound in Case 1. For future use, note that if $\eps_n \to 0$ is any sequence, 
and we add the restriction $n_1 \le \eps_n n$ to the conditions in Case 1, we obtain the 
upper bound
\begin{equation}
\begin{split}
  C |\bfK| \frac{\eps_n n}{N^d} p_{n}(y'_{\ell-1}, y'_\ell)
	= o(1) \frac{n}{N^d} p_{n}(y'_{\ell-1}, y'_\ell).
\end{split}
\label{e:Case1-epsn}
\end{equation}

\emph{Case 2a. $n_1, n_2 \ge N^{2+6\eps}$ but $|x' - y'_{\ell-1}| > n_1^{\frac{1}{2}+\eps}$.}
In this case we bound $p_{n_2}(x', y'_\ell) \le 1$ and have that the contribution of this case to the
right-hand side of \eqref{e:sum-A(x)-ub} is at most
\begin{equation*}
\begin{split}
 \sum_{\substack{n_1+n_2 = n \\ n_1,n_2 \ge N^{2+6\eps}}} 
   \P_{y'_{\ell-1}} [ |Y_{n_1} - y'_{\ell-1}| > n_1^{1/2+\eps} ] 
 &\le \sum_{\substack{n_1+n_2 = n \\ n_1,n_2 \ge N^{2+6\eps}}}
   C \exp ( - c n_1^{2\eps} )\\
 &\le C n \exp ( - c N^{4 \eps} )
 = o \left(\frac{n}{N^d}\right)  p_{n}(y'_{\ell-1}, y'_\ell), 
\end{split}
\end{equation*}
where in the first step we used \eqref{e:maximal_ineq} and 
in the last step we used the Gaussian lower bound \eqref{e:gaussian_ub_lb} for $p_n$.
Indeed, the requirement for the Gaussian lower bound is satisfied for sufficiently large $N$ because $|y'_{\ell}-y'_{\ell+1}| \le C_1 \sqrt{\log n}\sqrt{n} \le c\,n$.
Therefore, we have
%the Gaussian lower bound for $p_n$, which is applicable for $N$ large enough, we have
\begin{equation}
\label{e:p_n-lbd}
 p_{n}(y'_{\ell-1},y'_{\ell})
 \ge \frac{c}{n^{d/2}}\exp\left(-\frac{C|y'_{\ell}-y'_{\ell-1}|^2}{n}\right)
  \ge \frac{c}{n^{d/2}}\exp\left(-C\,\log n\right).
\end{equation}
Then we have  
\begin{equation*}
\begin{split}
\frac{Cn\exp (-cN^{4\eps})}{c n^{-d/2}\exp\left(-C\,\log n\right)}
\le Cn^{1+d/2}\exp\left(- c N^{4\eps} + C\,\log n\right)
= o \left( \frac{n}{N^d} \right), \,\text{as $N\to\infty$.}
\end{split}
\end{equation*}

\emph{Case 2b. $n_1, n_2 \ge N^{2+6\eps}$ but $|y'_{\ell} - x'| > n_2^{1/2+\eps}$.}
This case can be handled very similarly to Case 2a.

\emph{Case 3a. $n_1 < N^{2+6\eps}$ 
and $|x'-y'_{\ell-1}|\le N^{\frac{\delta}{2}-\eps}$.}
By the LCLT we have 
\begin{equation*}
\begin{split}
 p_{n_2}(x',y'_{\ell}) 
 &= \frac{C}{n_2^{d/2}}\exp\left(-\frac{d|y'_{\ell}-x'|^2}{n_2}\right)
 (1+o(1))\\
 p_{n}(y'_{\ell-1},y'_{\ell})
 &= \frac{C}{n^{d/2}}\exp\left(-\frac{d|y'_{\ell}-y'_{\ell-1}|^2}{n}\right)
 (1+o(1)).
\end{split}
\end{equation*}
We claim that 
\begin{equation}
\label{e:Case3a-compare}
p_{n_2}(x',y'_{\ell}) 
\le C\,p_n(y'_{\ell-1},y'_\ell).
\end{equation}
We first note that $n_2=n-n_1 = n(1+o(1))$, then we deduce that $n_2^{-d/2}= O(n^{-d/2})$
and 
\begin{equation*}
\exp\left(-\frac{d|y'_{\ell}-y'_{\ell-1}|^2}{n}\right)
\ge \exp\left(-\frac{d|y'_{\ell}-y'_{\ell-1}|^2}{n_2}\right).
\end{equation*}
Since we have $|x'-y'_{\ell-1}|\le N^{\frac{\delta}{2}-\eps}$ in the exponent, 
we have, as $N\to\infty$,
\begin{equation*}
\frac{|x'-y'_{\ell-1}|^2}{n_2}
\le \frac{N^{\delta-2\eps}}{n_2}
\to 0
\end{equation*}
and 
\begin{equation*}
\frac{|y'_{\ell}-y'_{\ell-1}||x'-y'_{\ell-1}|}{n_2}
\le \frac{n^{\frac{1}{2}}C_1\sqrt{\log n}N^{\frac{\delta}{2}-\eps}}{n_2}
\to 0.
\end{equation*}
These imply that
\begin{equation*}
\begin{split}
 \left|\frac{|y'_{\ell}-y'_{\ell-1}|^2-|y'_{\ell}-x'|^2}{n_2}\right|
 &\le \left|\frac{|y'_{\ell}-y'_{\ell-1}|^2
 	-|(y'_{\ell}-y'_{\ell-1})+(y'_{\ell-1}-x')|^2}{n_2}\right|\\
 &\le \frac{|x'-y'_{\ell-1}|^2}{n_2}
 	+\frac{2|y'_{\ell}-y'_{\ell-1}||x'-y'_{\ell-1}|}{n_2}
 \to 0.
\end{split} 
\end{equation*}
Thus \eqref{e:Case3a-compare} follows from comparing the LCLT 
approximations of the two sides.

We now have that the contribution of this case to the
right-hand side of \eqref{e:sum-A(x)-ub} is at most
\begin{equation*}
\begin{split}
 C p_{n}(y'_{\ell-1},y'_{\ell})
 \sum_{\substack{n_1 < N^{2+6\eps}}} \sum_{x \in \Z^d} \sum_{x'\in \bfK+xN} p_{n_1}(y'_{\ell-1},x')
 &\le \frac{C}{N^{\zeta(d-2)}} p_{n}(y'_{\ell-1},y'_{\ell}) \\
% &\le \frac{C}{N^{\frac{\delta}{d}(d-2)+\eps}} p_{n}(y'_{\ell-1},y'_{\ell})
 &\le o(1) \frac{n}{N^d} p_{n}(y'_{\ell-1},y'_{\ell}),
\end{split}
\end{equation*}
where in the first step we used Lemma \ref{lem:Green-bnd}(i) and 
the last step holds for the value of $\zeta$ we chose; cf.~\eqref{e:choice-zeta}.

\emph{Case 3b. $n_1 < N^{2+6\eps}$ but $|x'-y'_{\ell-1}|> N^{\frac{\delta}{2}-\eps}$.}
Use the Gaussian upper bound \eqref{e:gaussian_ub_lb} to bound $p_{n_1}$ and bound the sum over all $x'\in\Z^d$ of $p_{n_2}$ by 1 using symmetry of $p_{n_2}$ to get
\begin{equation*}
\begin{split}
 &\sum_{\substack{n_1+n_2=n\\n_1 < N^{2+6\eps}}} \sum_{x \in \Z^d} \sum_{x'\in\bfK+xN}
  p_{n_1}(y'_{\ell-1},x')p_{n_2}(x',y'_\ell)\\
  &\qquad\le \sum_{n_1 < N^{2+6\eps}} \frac{C}{n_1^{d/2}}\exp\left(-\frac{N^{\delta-2\eps}}{N^{2+6\eps}}\right) \sum_{x'\in \Z^d} p_{n-n_1}(x',y'_\ell)\\
 &\qquad\le \sum_{n_1 < N^{2+6\eps}} \frac{C}{n_1^{d/2}}\exp\left(-\frac{N^{\delta-2\eps}}{N^{2+6\eps}}\right)\\
 &\qquad \le C N^{O(1)} \exp(-N^{\delta-2-8\eps})
 = o\left(\frac{n}{N^d}\right) p_{n}(y'_{\ell-1}, y'_\ell), \quad\text{as $N\to\infty$.}
\end{split}
\end{equation*}
In the last step we used a Gaussian lower bound for $p_n$; cf.~\eqref{e:p_n-lbd}.

\emph{Case 4a. $n_2 < N^{2+6\eps}$ 
and $|y'_\ell - x'|\le N^{\frac{\delta}{2}-\eps}$.}  
This case can be handled very similarly to Case 3a.

\emph{Case 4b. $n_2 < N^{2+6\eps}$ 
and $|y'_\ell - x'|> N^{\frac{\delta}{2}-\eps}$.}
This case can be handled very similarly to Case 3b.

Therefore, we discussed all possible cases and proved statement 
\eqref{e:weaker_capacity-ub} of the proposition as required.

% Proof of the first stretch part

The proof of \eqref{e:weaker_capacity-ub-o} is similar to the first part 
with only a few modifications.
In this part we have to show that 
\begin{equation*}
 \P_{o} 
	 \left[ Y_n = y'_1,\, \text{$Y_t \in \varphi^{-1}(\mathbf{K})$ for some
	 $0 \le t < n$} \right]
 = o(1)\P_{o} \left[ Y_n = y'_1 \right].
\end{equation*}
Define $A_0(x)= \left\{Y_n=y'_1,
\text{ $Y_t\in x N + \bfK$ for some $0\leq t < n$} \right\}$, 
so that 
\begin{equation}
\label{e:sum-A_0(x)-ub}
 \mathbf{P}_{o} 
	 \left[ Y_n = y'_1,\, \text{$Y_t \in \varphi^{-1}(\mathbf{K})$ for some
	 $0 \le t < n$} \right] 
 \le \sum_{x \in \Z^d} \P_{o} [ A_0(x) ]. 
\end{equation}
We have
\begin{equation}
\label{e:A_0(x)-ub}
\begin{split}
\P_{o}[A_0(x)]
	\leq \sum_{n_1+n_2=n}\sum_{x'\in \bfK+xN}
		p_{n_1}(o,x') p_{n_2}(x',y'_{1}).
\end{split}
\end{equation}

%We bound above by splitting up the sum into different contributions. The different cases can be handled very similarly to the first part.
%
%\color{red}
We bound the term above by splitting up the sum into the same cases as in the proof of \eqref{e:weaker_capacity-ub}. The different cases can be 
handled very similarly to the first part. The difference is only in Case 3a while applying the Green's function bound Lemma \ref{lem:Green-bnd}.

In Case 3a, by the LCLT, we can deduce that
\begin{equation*}
   p_{n_2}(x',y'_1) 
    \le C\,p_n(o,y'_1). 
\end{equation*}
If $g(N) > N^\zeta$, the bound of Lemma \ref{lem:Green-bnd}(i) can be used as before.
If $g(N) \le N^\zeta$, by Lemma \ref{lem:Green-bnd}(ii), we have that the contribution of this case to the
right-hand side of \eqref{e:sum-A_0(x)-ub} is at most
\begin{equation*}
\begin{split}
 C p_{n}(o,y'_1)
 \sum_{\substack{n_1 < N^{2+6\eps}}} \sum_{x \in \Z^d} \sum_{x'\in \bfK+xN} p_{n_1}(o,x')
 \le \frac{C}{g(N)^{d-2}} p_{n}(o,y'_1) 
 = o(1) p_{n}(o,y'_1).
\end{split}
\end{equation*}
Here we used that $g(N) \to \infty$.

Note that Case 4a can be handled in the same way as in the proof of \eqref{e:weaker_capacity-ub},
since the distance between $y'_1$ and any copy of $\bfK$ is at least $N^\zeta$.

Therefore, we discussed all possible cases and proved \eqref{e:weaker_capacity-ub-o} as required.
\end{proof} 

For future use, we extract a few corollaries of the proof of Proposition \ref{prop:weaker_capacity_ub}.

\begin{corollary}
\label{cor:x-summation}
Assume that $y', y'' \in \Z^d$ are points such that 
$|y'' - y'| \le 2 C_1 \sqrt{\log n} \sqrt{n}$. 
Then for all $n/2 \le m \le n$ we have 
\begin{equation}
\label{e:x-summation}
    \sum_{n_1+n_2=m} \sum_{x \in \Z^d} \sum_{\substack{x'\in \bfK+xN \\ |y' - x'|, |y'' - x'| > N^\zeta}}
		p_{n_1}(y',x') p_{n_2}(x',y'')
    = O\left(\frac{n}{N^d}\right) p_{m}(y',y''). 
\end{equation}
\end{corollary}

\begin{proof}
In the course of the proof of Proposition \ref{prop:weaker_capacity_ub}, 
we established the above with $m = n$, where $y' = y'_{\ell-1}$ and $y'' = y'_{\ell}$, 
and with $C_1$ in place of $2 C_1$ in the upper bound on the displacement $|y'' - y'|$. 
Note that in this case the restriction $|y' - x'|, |y'' - x'| > N^\zeta$ in the summation
holds for all $x$, due to conditions imposed on $y'_{\ell-1}$ and $y'_{\ell}$ in the
definition of $\mathcal{G}_{\zeta,C_1}$.

The arguments when $n/2 \le m < n$ and with the upper bound increased by a factor of $2$ are essentially the same.
The information that $y'_{\ell-1}$ and $y'_\ell$ are at least distance $N^\zeta$ from $\varphi^{-1}(\bfK)$
was only used in Cases 3a and 4a to handle terms $x'$ close to these points.
Since in \eqref{e:x-summation} we exclude such $x'$ from the summation, the statement 
follows without restricting the location of $y', y''$.
\end{proof}

The following is merely a restatement of what was observed in \eqref{e:Case1-epsn} (with part (ii) 
below holding by symmetry). 

\begin{corollary}
\label{cor:Case1-extract}
(i) For $\ell \ge 2$ and any sequence $\eps_n \to 0$, we have
\begin{equation}
\begin{split}
  \sum_{\substack{n_1+n_2=n \\ n_1 \le \eps_n n \\ n_1, n_2 \ge N^{2+6\eps}}} \ 
  \sum_{x \in \Z^d} \ \sum_{\substack{x' \in \bfK +xN : \\ |y'_{\ell-1}-x'|\leq n_1^{\frac{1}{2}+\eps} \\
      |x'-y'_{\ell}|\leq n_2^{\frac{1}{2}+\eps}}}
      p_{n_1}(y'_{\ell-1},x') p_{n_2}(x',y'_{\ell})
  = o(1) \frac{n}{N^d} p_n(y'_{\ell-1}, y'_\ell).
\end{split}
\label{e:Case1-extract}
\end{equation}
(ii) The same right-hand side expression is valid if we replace the restriction $n_1 \le \eps_n n$ by $n_2 \le \eps_n n$.
\end{corollary}

The following is a restatement of the bounds of Cases 2a and 2b.

\begin{corollary}
\label{cor:Case2-extract}
For $\ell \ge 2$ we have \\
(i) \begin{equation}
\begin{split}
  \sum_{\substack{n_1+n_2=n \\ n_1, n_2 \ge N^{2+6\eps}}} \ 
  \sum_{x \in \Z^d} \ \sum_{\substack{x' \in \bfK +xN : \\ |y'_{\ell-1}-x'| > n_1^{\frac{1}{2}+\eps}}}
      p_{n_1}(y'_{\ell-1},x') p_{n_2}(x',y'_{\ell})
  = o(1) \frac{n}{N^d} p_n(y'_{\ell-1}, y'_\ell).
\end{split}
\label{e:Case2a-extract}
\end{equation}
(ii) \begin{equation}
\begin{split}
  \sum_{\substack{n_1+n_2=n \\ n_1, n_2 \ge N^{2+6\eps}}} \ 
  \sum_{x \in \Z^d} \ \sum_{\substack{x' \in \bfK +xN : \\ |x' - y'_\ell| > n_2^{\frac{1}{2}+\eps}}}
      p_{n_1}(y'_{\ell-1},x') p_{n_2}(x',y'_{\ell})
  = o(1) \frac{n}{N^d} p_n(y'_{\ell-1}, y'_\ell).
\end{split}
\label{e:Case2b-extract}
\end{equation}
\end{corollary}

The following is a restatement of the bounds of Cases 3a and 3b combined.

\begin{corollary}
\label{cor:Case3-extract}
For $\ell \ge 2$ we have
\begin{equation}
\begin{split}
  \sum_{\substack{n_1+n_2=n \\ n_1 < N^{2+6\eps}}} \ 
  \sum_{x \in \Z^d} \ \sum_{\substack{x' \in \bfK +xN}}
      p_{n_1}(y'_{\ell-1},x') p_{n_2}(x',y'_{\ell})
  = o(1) \frac{n}{N^d} p_n(y'_{\ell-1}, y'_\ell).
\end{split}
\label{e:Case3-extract}
\end{equation}
\end{corollary}

\subsection{Proof of Proposition \ref{prop:capacity_ub}}
\label{ssec:capacity_ub}

%\textbf{Sketch of the proof of \eqref{e:capacity_ub}.} \emph{(Last modified 09/6/2021)}

%We fix $2 < \delta < d$ with the property that 
%\[ \frac{\delta}{d} (2 - d) < \delta - d, \]
%which is possible, because this is equivalent to 
%\[ \delta > \frac{d^2}{d-1} \frac{1}{2}, \]
%and the right-hand side is between $2$ and $d$ for all $d \ge 3$.
%
%Next we fix $\zeta < \delta/d$, with $\zeta$ close enough to 
%$\delta/d$ so that we also have
%\[ \zeta (2 - d) < \delta - d. \]
%
%% Next we fix $\zeta < \zeta_2 < \delta/d$, with $\zeta$ close enough to 
%% $\delta/d$ so that we also have
%% \[ \zeta (2 - d) < \delta - d. \]

In this section we need $C_1$ large enough so that we have
\begin{equation}
\label{e:C_1-choice}
 e^{-d f(n)^2} N^d n^{1+3d/2}
 \to 0.
\end{equation}

We have
\[ \mathbf{P}_{y'_{\ell-1}} [ Y_n = y'_\ell,\, \text{$Y_t \in \varphi^{-1}(\mathbf{K})$ for some $0 \le t < n$} ]
  = \mathbf{P}_{y'_{\ell-1}} \left[ \cup_{x \in \mathbb{Z}^d} A(x) \right], \]
where
\[ A(x)
  = \left\{ Y_n = y'_\ell, \text{$Y_t \in x N + \mathbf{K}$ for some $0 \le t < n$} \right\}. \]
The strategy is to estimate the probability via the Bonferroni inequalities:
\begin{equation}
\label{e:Bonferroni}
\begin{split}
  \sum_{x} \mathbf{P}_{y'_{\ell-1}} [ A(x) ]
    - \sum_{x_1 \not= x_2} \mathbf{P}_{y'_{\ell-1}} [ A(x_1) \cap A(x_2) ]
	&\le \mathbf{P}_{y'_{\ell-1}} \left[ \cup_{x \in \mathbb{Z}^d} A(x) \right] \\
	&\le \sum_{x} \mathbf{P}_{y'_{\ell-1}} [ A(x) ]. 
\end{split}
\end{equation}

We are going to use a parameter $A_n$ that we choose as $A_n = 10 \log \log n$ 
so that in particular $A_n \to \infty$.

\subsubsection{The main contribution}

In this section, we consider only stretches with $|y'_\ell - y'_{\ell-1}| \le A_n \sqrt{n}$.
We will show that the main contribution in \eqref{e:Bonferroni} comes from $x$ in the set:
\[ G 
   = \left\{ x \in \mathbb{Z}^d : |y'_{\ell-1} - x N| \le A_n^2 \sqrt{n},\, 
     |x N - y'_{\ell}| \le A_n^2 \sqrt{n} \right\}. \]
We first examine $\P_{y'_{\ell-1}} [ A(x) ]$ for $x \in G$. 
Putting $B_{0,x} = B(\mathbf{x}_0 + x N, N^{\zeta})$, 
let $n_1$ be the time of the last visit to $\partial B_{0,x}$ before hitting 
$\mathbf{K}+xN$, let $n_1 + n_2$ be the time of the first hit of 
$\mathbf{K}+xN$, and let $n_3 = n - n_1 - n_2$. 
See Figure \ref{fig:decomposition} for an illustration of this decomposition.
%\vspace{2cm}
\begin{figure}%
\begin{center}
% --- Figure starts ---
\begin{tikzpicture}
\draw (0,0) circle (3cm);
\node[anchor=west] at (2.2,2.2) {$B_{0,x}$};
\draw[fill=black] plot[smooth, tension=.7] 
coordinates {(0.25,0) (0,-0.5) (-0.5,-0.3) (-0.3,0) (0,0.3) (0.2,0.2) (0.25,0)};
\node[anchor=west] at (0.2,-0.2)  {$\mathbf{K}+xN$};
\filldraw[black] (0,0.3) circle (1pt);
\node[anchor=west] at (0,0.5) {$x^\prime$};
\filldraw[black] (-2.12, 2.12) circle (1pt);
\node[anchor=west] at (-2.2,2.5) {$z^\prime$};

\filldraw[black] (-6, -1) circle (1pt);
\node[anchor=west] at (-6,-1) {$y_{\ell-1}^\prime$};
\filldraw[black] (6, -2.5) circle (1pt);
\node[anchor=west] at (6,-2.5) {$y_{\ell}^\prime$};

\color{green}
\draw (-6,-1) .. controls (-5.8,-0.5) and (-5.6,-1) .. (-5.5,-0.5)
      (-5.5,-0.5) .. controls (-5.4,-0.5) and (-5.2,-1) .. (-5.0,-0.5)
      (-5.0,-0.5) .. controls (-4.8,0) and (-4.6,-0.5) .. (-4.5,-0.75)
      (-4.5,-0.75) .. controls (-4.4,-0.5) and (-4.2,0) .. (-4.0,-0.5)
      (-4.0,-0.5) .. controls (-3.8,0) and (-3.6,0.5) .. (-3.5,0.5)
      (-3.5,0.5) .. controls (-3.4,0.5) and (-3.2,1) .. (-3.0,0.5)
      (-3.0,0.5) .. controls (-2.8,1) and (-2.6,0.5) .. (-2.5,1) 
      (-2.5,1) .. controls (-2.4,0.75) and (-2.2,1) .. (-2.0,1.5)
      (-2.0,1.5).. controls (-1.5,1) and (-2,1.5) .. (-2.12, 2.12);

\color{red}  
\draw (-2.12, 2.12) .. controls (-2,2) and (-1.75,1.5) .. (-1.5,1.75)
	  (-1.5, 1.75) .. controls (-1.25,1.5) and (-1, 1) .. (-0.75,0.75)
	  (-0.75, 0.75) .. controls (-0.5,1.5) and (-0.25,0.5) .. (0,0.3);

\color{blue}	  
\draw (0,0.3) .. controls (0.5,0.5) and (0.75,0) .. (1,0.5)
	  (1,0.5) .. controls (1.25,-0.5) and (1.5,-1) .. (1.75,-0.5)
	  (1.75,-0.5) .. controls (2,0) and (2.25,-0.5) .. (2.5,-1)
	  (2.5,-1).. controls (2.6,-0.5) and (2.8,-1) .. (3,-0.5)
      (3,-0.5) .. controls (3.2,-0.5) and (3.4,-1) .. (3.5,-0.5)
      (3.5,-0.5) .. controls (3.6,0) and (3.8,-0.5) .. (4.0,-0.75)
      (4.0,-0.75) .. controls (4.2,-1) and (4.4,-0.5) .. (4.5,-0.75)
      (4.5,-0.75) .. controls (4.6,-1) and (4.8,-1.5) .. (5.0,-2)
      (5.0,-2) .. controls (5.2,-1.5) and (5.4,-1) .. (5.5,-1.5)
      (5.5,-1.5) .. controls (5.6,-1.75) and (5.8,-2) .. (6.0,-2.5);

\color{black}      
\end{tikzpicture}
% --- Figure ends ---
\caption{The decomposition of a path hitting a copy of $\bfK$ into three subpaths (not to scale).}%
\label{fig:decomposition}%
\end{center}
\end{figure}
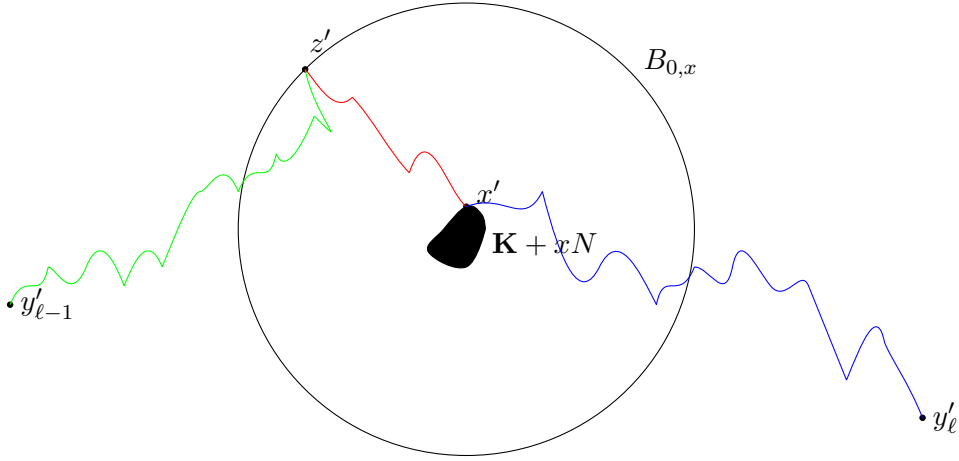
%\vspace{2cm}
Then we can write:
\begin{equation}
\label{e:decomp}
\begin{split}
  \P_{y'_{\ell-1}} [ A(x) ]
	&= \sum_{n_1 + n_2 + n_3 = n} \sum_{z' \in \partial B_{0,x}} \sum_{x' \in \mathbf{K} + x N}
	   \widetilde{p}^{(x)}_{n_1} (y'_{\ell-1},z') \\
	&\qquad\quad \times \mathbf{P}_{z'} [ H_{\mathbf{K}+xN} = n_2 < \xi_{B_{0,x}},\, Y_{H_{\mathbf{K}+xN}} = x' ] \,
		 p_{n_3}(x',y'_\ell), 
\end{split}
\end{equation}
where 
\[ \widetilde{p}^{(x)}_{n_1} (y'_{\ell-1},z')
   = \mathbf{P}_{y'_{\ell-1}} [ Y_{n_1} = z',\, \text{$Y_t \not\in \mathbf{K}+xN$ for $0 \le t \le n_1$} ]. \]

We are going to use another parameter $\eps_n$ that will need to 
go to $0$ slowly. We choose it as $\eps_n = (10 \log \log n)^{-1} \to 0$.
The main contribution to \eqref{e:decomp} will be when 
$n_1 \ge \eps_n n$, $n_3 \ge \eps_n n$ and $n_2 \le N^{2\delta/d} \sim n^{2/d}$.
Therefore, we split the sum over $n_1, n_2, n_3$ in \eqref{e:decomp} into a
main contribution $I(x)$ and an error term $II(x)$.
In order to define these, let
\begin{equation}
\begin{split}
\label{e:F-def}
    &F(n_1,n_2,n_3,x,y'_{\ell-1},y'_\ell) \\
    &\quad = \sum_{z' \in \partial B_{0,x}} \sum_{x' \in \mathbf{K} + x N}\widetilde{p}^{(x)}_{n_1} (y'_{\ell-1},z')  \mathbf{P}_{z'} [ H_{\mathbf{K}+xN} = n_2 < \xi_{B_{0,x}},\, Y_{H_{\mathbf{K}+xN}} = x' ]
		 p_{n_3}(x',y'_\ell).
\end{split}
\end{equation}
Then with
\begin{equation}
\label{e:decomp-mod}
\begin{split}
    I(x)
    &:= \sum_{\substack{n_1 + n_2 + n_3 = n \\ n_1, n_3 \ge \eps_n n,\, n_2 \le N^{2 \delta/d}}}
	   F(n_1,n_2,n_3,x,y'_{\ell-1},y'_\ell) \\
	II(x)
	&:= \sum_{\substack{n_1 + n_2 + n_3 = n \\ \text{$n_1 < \eps_n n$ or $n_3 < \eps_n n$} \\ \text{or $n_2 > N^{2 \delta/d}$}}}
	   F(n_1,n_2,n_3,x,y'_{\ell-1},y'_\ell)
\end{split}
\end{equation}
we have
\begin{equation*}
\begin{split}
  \P_{y'_{\ell-1}} [ A(x) ]
	= I(x) + II(x).
\end{split}
\end{equation*}

\begin{lemma}
\label{lem:p_n3}
When $x \in G$ and $n_3 \ge \eps_n n$, we have
\[ p_{n_3}(x',y'_\ell)
   = (1 + o(1)) p_{n_3}(u',y'_\ell) \quad \text{for all $x' \in \mathbf{K}+xN$ 
	   and all $u' \in \mathbb{T}_N + xN$}, \]
with the $o(1)$ term uniform in $x'$ and $u'$.
\end{lemma}

\begin{proof}
By the LCLT, we have
\begin{equation*}
\begin{split}
 p_{n_3}(x',y'_\ell) 
 &= \frac{C}{n_3^{d/2}}\exp\left(-\frac{d|y'_{\ell}-x'|^2}{n_3}\right)
 (1+o(1)),\\ 
 p_{n_3}(u',y'_\ell) 
 &= \frac{C}{n_3^{d/2}}\exp\left(-\frac{d|y'_{\ell}-u'|^2}{n_3}\right)
 (1+o(1)). 
\end{split}
\end{equation*}
We compare the exponents
\begin{equation*}
\begin{split}
\left|\frac{d|y'_{\ell}-x'|^2}{n_3} - \frac{d|y'_{\ell}-u'|^2}{n_3}\right|
&\leq \frac{d|x'-u'|^2}{n_3} 
+ \frac{2d|\langle x'-u',y'_{\ell}-x'\rangle|}{n_3} \\
&\leq C\frac{N^2}{n_3} + \frac{CN\cdot A_n^2\sqrt{n}}{n_3}
\to 0,
\end{split}
\end{equation*}
as $N\to\infty$.
\end{proof}

\begin{lemma}
\label{lem:p_n1}
When $x \in G$ and $n_1 \ge \eps_n n$, we have
\[ \widetilde{p}^{(x)}_{n_1}(y'_{\ell-1},z')
   = (1 + o(1)) p_{n_1}(y'_{\ell-1},u') \quad \text{for all $z' \in \partial B_{0,x}$ 
	   and all $u' \in \mathbb{T}_N + xN$}, \]
with the $o(1)$ term uniform in $z'$ and $u'$.
\end{lemma}

\begin{proof}
The statement will follow if we show the following claim:
\[ \mathbf{P}_{y'_{\ell-1}} [ Y_{n_1} = z',\, \text{$Y_t \in \mathbf{K}+xN$ for some $0 \le t \le n_1$} ]
   = o(1) p_{n_1} (y'_{\ell-1}, z'). \]
For this, observe that by \eqref{e:gaussian_ub_lb} we have
\begin{equation}
\label{e:rhs-lb2}
\begin{split}
  p_{n_1}(y'_{\ell-1}, z')
  &\ge \frac{c}{n_1^{d/2}} \exp \left( -C \frac{|z' - y'_{\ell-1}|^2}{n_1} \right) \\
	&\ge \frac{c}{n^{d/2}} \exp \left( -c \frac{A_n^2 n + N^{\zeta}}{\eps_n n} \right) \\
  &\ge \frac{c}{n^{d/2}} \exp \left( -C (\log \log n)^{O(1)} \right) \\
	&= n^{-d/2+o(1)}.
\end{split}
\end{equation}
On the other hand, using the Markov property, \eqref{e:gaussian_ub_lb}, and the fact
that for $x' \in \mathbf{K}+xN$ we have $|y'_{\ell-1} - x'| \ge c N^{\zeta}$
and $|x' - z'| \ge c N^{\zeta}$, we get
\begin{equation}
\label{e:lhs-ub2}
\begin{split}
  &\mathbf{P}_{y'_{\ell-1}} [ Y_{n_1} = z',\, \text{$Y_t \in \mathbf{K}+xN$ for some $0 \le t \le n_1$} ] \\
	&\qquad \le \sum_{1 \le m \le n_1-1} \sum_{x' \in \mathbf{K}+xN} p_m(y'_{\ell-1},x') \, 
	    p_{n_1-m}(x',z') \\
	&\qquad \le C \sum_{1 \le m \le n_1-1} \frac{1}{m^{d/2}} 
	    \frac{1}{(n_1-m)^{d/2}} \exp \left( - c \frac{N^{2\zeta}}{m} \right)
			\exp \left( -c \frac{N^{2\zeta}}{n_1-m} \right)\\
	&\qquad \le C \sum_{1 \le m \le n_1/2} \frac{1}{m^{d/2}} 
	    \frac{1}{(n_1-m)^{d/2}} \exp \left( - c \frac{N^{2\zeta}}{m} \right)
			\exp \left( -c \frac{N^{2\zeta}}{n_1-m} \right).
\end{split}
\end{equation} 
We note here that the sum over $1 \le m \le n_1/2$ and the sum over $n_1/2 \le m \le n_1-1$ are symmetric.
Bounding the sum over $1 \le m \le n_1/2$ gives
\begin{equation}
\label{e:lhs-ub3}
\begin{split}
	&\frac{C}{n_1^{d/2}} \sum_{1 \le m \le n_1/2} \frac{1}{m^{d/2}} \exp \left( - c \frac{N^{2\zeta}}{m} \right) \\
    &\quad = \frac{C}{n_1^{d/2}} \left[ \sum_{1 \le m \le N^{2 \zeta}} \frac{1}{m^{d/2}}\exp \left( - c \frac{N^{2\zeta}}{m} \right) + \sum_{N^{2 \zeta} < m \le n_1/2} \frac{1}{m^{d/2}}\exp \left( - c \frac{N^{2\zeta}}{m} \right)\right]. 
\end{split}
\end{equation} 
In the second sum we can bound the exponential by $1$, and get the upper bound
\[ \frac{C}{n_1^{d/2}} N^{\zeta(2-d)} = o(n^{-d/2+o(1)}). \]
In the first sum, we group terms on dyadic scales $k$ so that 
$2^k \le N^{2 \zeta}/m \le 2^{k+1}$, $k = 0, \dots, \lfloor \log_2 N^{2 \zeta}\rfloor +1 $.
This gives the bound
\[ \frac{C}{n_1^{d/2}} \sum_{k=0}^{\lfloor\log_2 N^{2 \zeta}\rfloor +1 }
   \frac{(2^{k+1})^{d/2}}{(N^{2 \zeta})^{d/2}} \exp \left( -c 2^k \right)
	 \le \frac{C}{n_1^{d/2}} \frac{1}{N^{\zeta d}}, \] 
which is also $o(n^{-d/2+o(1)})$. 
%which is of the same order as the other term.
\end{proof}

In order to apply the previous two lemmas to analyze $I(x)$ in \eqref{e:decomp-mod}, we first define a modification of $F$ in \eqref{e:F-def}, where $z'$ and $x'$ are both replaced by a vertex $u' \in \T_N + xN$. That is, we define 
\begin{equation*}
\begin{split}
    &\widetilde{F}(n_1,n_2,n_3,u',x,y'_{\ell-1},y'_\ell) \\
    &\quad = \sum_{z' \in \partial B_{0,x}} \sum_{x' \in \mathbf{K} + x N} p_{n_1} (y'_{\ell-1},u')  \mathbf{P}_{z'} [ H_{\mathbf{K}+xN} = n_2 < \xi_{B_{0,x}},\, Y_{H_{\mathbf{K}+xN}} = x' ]
		 p_{n_3}(u',y'_\ell).
\end{split}
\end{equation*}
Then the Lemmas \ref{lem:p_n3} and \ref{lem:p_n1} allow us to write, for $x \in G$, the main term $I(x)$ in \eqref{e:decomp-mod} as
\begin{equation}
\label{e:decomp-2}
\begin{split}
  I(x)
   & = \sum_{\substack{n_1 + n_2 + n_3 = n \\ n_1, n_3 \ge \eps_n n,\, n_2 \le N^{2\delta/d}}} F(n_1,n_2,n_3,x,y'_{\ell-1},y'_\ell) \\
    & = \frac{1 + o(1)}{N^d} \sum_{u' \in \mathbb{T}_N+xN} \sum_{\substack{n_1 + n_2 + n_3 = n \\ n_1, n_3 \ge \eps_n n,\, n_2 \le N^{2 \delta/d}}} \widetilde{F}(n_1,n_2,n_3,u',x,y'_{\ell-1},y'_\ell)\\ 
	&= \frac{1 + o(1)}{N^d} \sum_{u' \in \mathbb{T}_N+xN} 
	   \sum_{\substack{n_1 + n_2 + n_3 = n \\ n_1, n_3 \ge \eps_n n \\ n_2 \le N^{2 \delta/d}}} 
	   p_{n_1} (y'_{\ell-1}, u') \, p_{n_3}(u', y'_\ell) \\
	&\qquad \times\sum_{z' \in \partial B_{0,x}} \mathbf{P}_{z'} [ H_{\mathbf{K}+xN} = n_2 < \xi_{B_{0,x}} ],
\end{split}
\end{equation}
where the sum over $x'$ is removed since
\begin{equation*}
    \sum_{x' \in \mathbf{K} + x N} \mathbf{P}_{z'} [ H_{\mathbf{K}+xN} = n_2 < \xi_{B_{0,x}},\, Y_{H_{\mathbf{K}+xN}} = x' ] 
    = \mathbf{P}_{z'} [ H_{\mathbf{K}+xN} = n_2 < \xi_{B_{0,x}} ].
\end{equation*}

\begin{lemma}
\label{lem:conv-and-sum}
Assume that $n_1, n_3 \ge \eps_n n$ and $n_2 \le N^{2 \delta/d}$.\\
(i) We have
\[ p_{n_1+n_3} (y'_{\ell-1}, y'_\ell)
   = (1 + o(1)) p_n(y'_{\ell-1}, y'_\ell). \]
(ii) We have
\begin{equation}
\label{e:sum-xinG}
   \sum_{x \in G} \sum_{u' \in \mathbb{T}_N+xN} 
      p_{n_1} (y'_{\ell-1}, u') \, p_{n_3}(u', y'_\ell)  
	 = (1 + o(1)) p_{n_1+n_3} (y'_{\ell-1}, y'_\ell). 
\end{equation}
\end{lemma}

\begin{proof}
(i) When $n_2 \le N^{2 \delta/d} \sim n^{2/d}$, we have
\[ n_1 + n_3
   = n \left( 1 - O \left( n^{-1+2/d} \right) \right). \]
Hence the exponential term in the LCLT for $p_{n_1+n_3}(y'_{\ell-1},y'_\ell)$ is
\[ \exp \left( - \frac{|y'_{\ell} - y'_{\ell-1}|^2}{n} \left( 1 + O ( n^{-1+2/d} ) \right) \right)
   = (1 + o(1)) \, \exp \left( - \frac{|y'_{\ell} - y'_{\ell-1}|^2}{n} \right), \]
where we used that $A_n = 10 \log \log n$, and hence  $|y'_{\ell} - y'_{\ell-1}|^2 \le A_n^2 n = n\, o( n^{1-2/d} )$. 

(ii) If we summed over all $x \in \mathbb{Z}^d$, we would get exactly $p_{n_1+n_3}(y'_{\ell-1}, y'_\ell)$.
Thus the claim amounts to showing that
\begin{equation}
\label{e:Gc-sum}
   \sum_{x \in \mathbb{Z}^d \setminus G} \sum_{u' \in \mathbb{T}_N+xN} 
      p_{n_1} (y'_{\ell-1}, u') \, p_{n_3}(u', y'_\ell)  
	 = o(1) p_{n_1+n_3} (y'_{\ell-1}, y'_\ell). 
\end{equation}
First, note that from the Local CLT we have
\[ p_{n_1+n_3} (y'_{\ell-1}, y'_\ell)
   = (1 + o(1)) \overline{p}_{n_1+n_3} (y'_{\ell-1}, y'_\ell). \]
In order to estimate the left-hand side of \eqref{e:Gc-sum}, using \eqref{e:gaussian_ub_lb}, 
the contribution of 
$\{ x \in \mathbb{Z}^d \setminus G : \max \{ |y'_{\ell-1} - xN|, |x N - y'_{\ell-1}| \} > A_n^2 \sqrt{n} \}$
can be estimated as follows. First, we have
\[ p_{n_1+n_3} (y'_{\ell-1}, y'_\ell)
   \ge \frac{c}{n^{d/2}} \exp ( - C A_n^2 (1+ o(1)) )
	 \ge \frac{c}{n^{d/2}} \exp ( - C (\log \log n)^2 ). \]
Here we used $|y'_{\ell} - y'_{\ell-1}|^2 \le A_n^2 n$ and $n_1+n_3 = n(1-o(1))$.

On the other hand, note that either $n_1 \ge n/3$ or $n_3 \ge n/3$.
Without loss of generality, assume that $n_3 \ge n/3$.
Then the contribution to the left-hand side of \eqref{e:Gc-sum}, 
using \eqref{e:gaussian_ub_lb}, and by summing in dyadic shells with radii
$2^k A_n^2 \sqrt{n}$, $k = 0, 1, 2, \dots$ we get the bound
\begin{equation}
\label{e:far-away-points}
\begin{split}
  &\sum_{k=0}^\infty C (A_n^2 \sqrt{n})^d 2^{dk} \, \frac{C}{n_1^{d/2}} \, \exp ( - c 2^{2k} A_n^4 n / n_1 ) \, 
	   \frac{1}{n_3^{d/2}} \, \exp ( - c 2^{2k} A_n^4 n / n_3 ) \\
	&\qquad \le \sum_{k=0}^\infty C \, A_n^{2d} \, 2^{dk} \, \frac{1}{\eps_n^{d/2}} \, \exp ( - c 2^{2k} A_n^4 ) 
	   \frac{1}{n^{d/2}} \, \exp ( - c 2^{2k} A_n^4 ) \\
	&\qquad \le \frac{C}{n^{d/2}} \, \frac{A_n^{2d}}{\eps_n^{d/2}} \, \sum_{k=0}^\infty 
	   \exp ( - c 2^{2k} (\log \log n)^4 + d k \log 2 ) \\
  &\qquad = \frac{C}{n^{d/2}} \, o \left( \exp ( - 100 (\log \log n)^2 ) \right). 
\end{split}	
\end{equation}

\end{proof}

The above lemma allows us to write 
\begin{equation}
\label{e:decomp-3}
\begin{split}
  \sum_{x \in G} I(x) 
	&= \frac{1 + o(1)}{N^d} \, p_n(y'_{\ell-1}, y'_\ell) \, 
	   \sum_{\substack{n_1 + n_2 + n_3 = n \\ n_1, n_3 \ge \eps_n n \\ n_2 \le N^{2 \delta/d}}} 
		 \sum_{z' \in \partial B_{0,x}} \mathbf{P}_{z'} [ H_{\mathbf{K}+xN} = n_2 < \xi_{B_{0,x}} ] \\
  &= \frac{(1 + o(1)) \, n}{N^d} \, p_n(y'_{\ell-1}, y'_\ell) \, 
	   \sum_{n_2 \le N^{2 \delta/d}} 
		 \sum_{z' \in \partial B_{0,x}} \mathbf{P}_{z'} [ H_{\mathbf{K}+xN} = n_2 < \xi_{B_{0,x}} ]. 
\end{split}
\end{equation}

The next lemma will help us extract the $\cpty(K)$ contribution from the 
right-hand side of \eqref{e:decomp}.

\begin{lemma}
We have
\begin{equation}
\label{e:capacity-2}
\begin{split}
  \sum_{n_2=0}^{N^{2 \delta/d}} \sum_{z' \in \partial B_{0,x}} 
	   \mathbf{P}_{z'} [ H_{\mathbf{K}+xN} = n_2 < \xi_{B_{0,x}} ] 
	 = \frac{1}{2} \cpty(K) \, (1 + o(1)). 
\end{split}
\end{equation} 
\end{lemma}

\begin{proof}
Performing the sum over $n_2$ allows us to re-write the expression in the left-hand side of \eqref{e:capacity-2} as
\begin{equation*}
\begin{split}
  &\left( \sum_{z' \in \partial B_{0,x}}
      \mathbf{P}_{z'} [ H_{\mathbf{K}+xN} < \xi_{B_{0,x}} ] \right)
	 - \left( \sum_{z' \in \partial B_{0,x}}
      \mathbf{P}_{z'} [ N^{2 \delta/d} < H_{\mathbf{K}+xN} < \xi_{B_{0,x}} ] \right) \\
	 &\qquad = \frac{1}{2} \cpty(K) +o(1) - 
      \sum_{z' \in \partial B_{0,x}} \sum_{\mathbf{x} \in \mathbf{K}} 
	    \mathbf{P}_{z'} [ N^{2 \delta/d} < H_{\mathbf{K}+xN} < \xi_{B_{0,x}},\, Y_{H_{\mathbf{K}+xN}} = \mathbf{x} + xN ].
\end{split}
\end{equation*}
Here the $1/2$ before $\cpty(K)$ comes from the random walk being lazy; see \eqref{e:capacity_ball}. Using time-reversal for the summand in the last term we get the expression
\begin{equation}
\label{e:capacity-3}
\begin{split}
  &= \frac{1}{2} \cpty(K) +o(1) - \sum_{\mathbf{x} \in \mathbf{K}} \sum_{z' \in \partial B_{0,x}} 
	    \mathbf{P}_{\mathbf{x}+xN} [ N^{2 \delta/d} < \xi_{B_{0,x}} < H_{\mathbf{K}+xN},\, 
        Y_{\xi_{B_{0,x}}} = z'] \\
 &\qquad = \frac{1}{2} \cpty(K) +o(1) - \sum_{\mathbf{x} \in \mathbf{K}} \mathbf{P}_{\mathbf{x}+xN} [ N^{2 \delta/d} < \xi_{B_{0,x}} < H_{\mathbf{K}+xN} ]. 
\end{split}
\end{equation}

The subtracted term in the right-hand side of \eqref{e:capacity-3} is at most
\[ |\bfK| \max_{\mathbf{x} \in \mathbf{K}} \mathbf{P}_{\mathbf{x}+xN} [ \xi_{B_{0,x}} > N^{2 \delta/d} ]. \]
Since $\zeta < \delta/d$, this expression is $o(1)$.
\end{proof}

From the above lemma we get that the main contribution equals
\begin{equation}
\label{e:main-contrib}
  \sum_{x \in G} I(x) 
	= (1 + o(1)) \frac{n}{N^d} \, \frac{1}{2} \, \cpty(K) \, p_n(y'_{\ell-1}, y'_\ell). 
\end{equation}

It is left to estimate all the error terms.

\subsubsection{The error terms}

\begin{lemma}
We have
\[ \sum_{x \in G} II(x) 
   = o(1) \frac{n}{N^d} p_{n}(y'_{\ell-1}, y'_\ell). \]
\end{lemma}

\begin{proof}
We split the estimates according to which condition is violated in the sum. Recall that in the proof of Proposition \ref{prop:weaker_capacity_ub} we chose $\eps > 0$ such that $-\eps+6\eps^2 <0$. Here we make the further restriction that $\eps < 2\delta/d - 2\zeta$.

\medbreak

\emph{Case 1. $n_2 > N^{2\delta/d}$.}
We claim that
\begin{equation}
\label{e:n2-claim}
 \P_{z'} [ H_{\bfK+xN} = n_2 < \xi_{B_{0,x}}, Y_{H_{\bfK+xN}} = x' ]
 \le C \exp ( - N^{\eps/2} ) p_{n_2}(z',x'). 
\end{equation}
Since in every time interval of duration $N^{2\zeta}$, the walk has a positive chance to exit the ball $B_{0,x}$, we have
\begin{equation*}
\begin{split}
 \P_{z'} [ H_{\bfK+xN} = n_2 < \xi_{B_{0,x}}, Y_{H_{\bfK+xN}} = x' ]
 &\le \P_{z'}[ \xi_{B_{0,x}} > N^{2\delta /d}]
 \le C \exp(-c \frac{N^{2\delta /d}}{N^{2\zeta}} ) \\
 &\le C \exp ( - N^{\eps} ).
\end{split}
\end{equation*}
By \eqref{e:gaussian_ub_lb} on $p_{n_2}$ and since $\zeta < \delta/d$ and $ N^{2\delta /d} < n_2 < n$ we have
\begin{equation*}
p_{n_2}(z',x')
\ge \frac{c}{n_2^{d/2}}\exp\left(-C\frac{N^{2\zeta}}{n_2}\right)
\ge c\exp ( - N^{\eps/2}).
\end{equation*}
Here we lower bounded 
$\exp\left(-C\frac{N^{2\zeta}}{n_2}\right)$ by $c$. The claim \eqref{e:n2-claim} is proved.

We also have the bound
\begin{equation*}
 \widetilde{p}_{n_1}^{(x)}(y'_{\ell-1},z')
 \le p_{n_1}(y'_{\ell-1},z').
\end{equation*}
We then get (summing over $z'$ and $x'$) that the contribution to $\sum_{x \in \Z^d} II(x)$ 
from Case 1 is at most
\begin{equation*}
\begin{split}
 &\sum_{n_1 + n_2 + n_3 = n} \sum_{z' \in \Z^d} \sum_{x' \in \Z^d}
   p_{n_1}(y'_{\ell-1},z') \, C \exp ( - N^{\eps/2} ) p_{n_2}(z',x') \, p_{n_3} (x',y'_{\ell}) \\
 &\qquad \le C \exp ( - N^{\eps/2} ) \sum_{n_1 + n_2 + n_3 = n} p_{n}(y'_{\ell-1}, y'_\ell) \\
 &\qquad \le C n^2 \exp ( - N^{\eps/2} ) p_{n}(y'_{\ell-1}, y'_\ell) \\
 &\qquad = o(1) \frac{n}{N^d} p_{n}(y'_{\ell-1}, y'_\ell).
\end{split}
\end{equation*}

\medbreak

\emph{Case 2. $n_2 \le N^{2\delta/d}$ and $n_1 < \eps_n n$.} Note that since $n_2 \le N^2 \le \eps_n n$
for large enough $N$, if we put $n'_1 = n_1 + n_2$ and $n'_2 = n_3$, we can upper bound the contribution of this case by 
\begin{equation*}
 \sum_{\substack{n'_1 + n'_2 = n \\ n'_1 \le 2 \eps_n n}} \sum_{x \in \Z^d} \sum_{x' \in \bfK+xN} 
   p_{n'_1}(y'_{\ell-1}, x') p_{n'_2}(x', y'_\ell). 
\end{equation*}
Now we can make use of the corollaries stated after the proof of Proposition \ref{prop:weaker_capacity_ub} as follows.

\emph{Case 2--(i). $N^{2+6 \eps} \le n'_1 \le 2 \eps_n n$ and $|y'_{\ell-1} - x'| \le (n'_1)^{\frac{1}{2}+\eps}$ 
and $|x' - y'_\ell| \le (n'_2)^{\frac{1}{2}+\eps}$.} 
Note that for large enough $N$ we have $n'_2 \ge (n - 2 \eps_n n) \ge N^{2+6 \eps}$.
Hence due to Corollary \ref{cor:Case1-extract}(i) (with $\eps_n$ there
replaced by $2 \eps_n$) the contribution of this case is 
%The LCLT allows us to replace 
%$x'$ by $u' \in \T+xN$ both in $p_{n'_1}$ and $p_{n'_3}$, yielding the upper bound
\begin{equation*}
\begin{split}
% \frac{C}{N^d} \sum_{\substack{n'_1 + n'_3 = n \\ n'_1 \le 2 \eps_n n}} \sum_{u' \in \Z^d} 
%   p_{n'_1}(y'_{\ell-1}, u') p_{n'_3}(u', y'_\ell)
% &= \frac{C}{N^d} \sum_{\substack{n'_1 + n'_3 = n \\ n'_1 \le 2 \eps_n n}} 
%   p_{n}(y'_{\ell-1}, y'_\ell)
% = C \frac{2 \eps_n n}{N^d} p_{n}(y'_{\ell-1}, y'_\ell)\\
  o(1) \frac{n}{N^d} p_{n}(y'_{\ell-1}, y'_\ell).
\end{split}
\end{equation*}

\emph{Case 2--(ii). $N^{2+6 \eps} \le n'_1 \le 2 \eps_n n$ but either 
$|y'_{\ell-1} - x'| > (n'_1)^{\frac{1}{2}+\eps}$ or $|x' - y'_\ell| > (n'_2)^{\frac{1}{2}+\eps}$.}
Again, we have $n'_2 \ge N^{2+6 \eps}$. Hence, neglecting the requirement $n'_1 \le 2 \eps_n n$,
Corollary \ref{cor:Case2-extract} immediately implies that the contribution of this case is
\begin{equation*}
\begin{split}
  o(1) \frac{n}{N^d} p_{n}(y'_{\ell-1}, y'_\ell).
\end{split}
\end{equation*}

\emph{Case 2--(iii). $n'_1 < N^{2+6\eps}$.} It follows immediately from Corollary \ref{cor:Case3-extract}
that the contribution of this case is
\begin{equation*}
\begin{split}
  o(1) \frac{n}{N^d} p_{n}(y'_{\ell-1}, y'_\ell).
\end{split}
\end{equation*}

\medbreak

\emph{Case 3. $n_2 \le N^{2\delta/d}$ and $n_3 < \eps_n n$.} Due to symmetry, this case can be 
handled very similarly to Case 2.
\end{proof}

\begin{lemma}
We have
\[ \sum_{x \in \mathbb{Z}^d \setminus G} \mathbf{P} [ A(x) ]
   = o(1) \frac{n}{N^d} p_n(y'_{\ell-1}, y'_\ell). \]
\end{lemma}

\begin{proof}
By the same arguments as in Lemma \ref{lem:conv-and-sum}(ii), we have
\[ p_n(y'_{\ell-1}, y'_\ell)
   \ge \frac{C}{n^{d/2}} \exp \left( - 100 (\log \log n)^2 \right). \]
For $x \in \mathbb{Z}^d \setminus G$, let $k$ be the dyadic scale 
that satisfies
\[ 2^k A_n^2 \sqrt{n}
   \le |x' - y'_{\ell-1}| 
	 < 2^{k+1} A_n^2 \sqrt{n}. \]
The same bounds hold up to constants for $|x' - y'_{\ell}|$.

Then we have
\begin{equation*}
\begin{split}
  \mathbf{P} [ A(x) ]
	&\le \sum_{1 \le m \le n-1} \sum_{x' \in \mathbf{K}+xN} p_m(y'_{\ell-1}, x') p_{n-m}(x',y'_\ell) \\
  &\le C |\bfK| \sum_{1 \le m \le n-1} \frac{1}{m^{d/2}} \frac{1}{(n-m)^{d/2}} 
	    \exp \left( -c \frac{2^{2k} A_n^4 n}{m} \right)
			\exp \left( -c \frac{2^{2k} A_n^4 n}{n-m} \right). 
\end{split}
\end{equation*}
Due to symmetry of the right-hand side, it is enough to consider the contribution 
of $1 \le m \le n/2$, which is bounded by 
\begin{equation*}
\begin{split}
  &\frac{C}{n^{d/2}} \exp \left( -c 2^{2k} A_n^4 \right) 
	  \sum_{1 \le m \le n/2} \frac{1}{m^{d/2}} \exp \left( -c \frac{2^{2k} A_n^4 n}{m} \right) \\
  &\qquad \le \frac{C}{n^{d/2}} \exp \left( -c 2^{2k} A_n^4 \right) 
	  \sum_{k' = 1}^{\lfloor \log_2 n \rfloor} \sum_{m : 2^{k'} \le n/m < 2^{k'+1}}
	  \frac{2^{k' d/2}}{n^{d/2}} \exp \left( -c 2^{2k} A_n^4 2^{k'} \right) \\
	&\qquad \le \frac{C}{n^{d}} \exp \left( -c 2^{2k} A_n^4 \right) 
	  \sum_{k' = 1}^{\infty} \frac{n}{2^{k'}}
	  \exp \left( -c 2^{2k} A_n^4 2^{k'} + k' d/2 \log 2 \right) \\
	&\qquad \le \frac{C n}{n^{d}} \exp \left( -c 2^{2k} A_n^4 \right). 
\end{split}
\end{equation*}

Now summing over $x \in \mathbb{Z}^d \setminus G$ we have that
the number of the copies of the torus at dyadic scale $2^k A_n^2 \sqrt{n}$ is at most $C \frac{1}{N^d} \left( 2^k A_n^2 \sqrt{n} \right)^d$. Hence
\begin{equation*}
\begin{split}
  \sum_{x \in \mathbb{Z}^d \setminus G} \mathbf{P} [ A(x) ]
	&\le \frac{C n}{n^{d}} \sum_{k=0}^\infty \frac{1}{N^d} \left( 2^k A_n^2 \sqrt{n} \right)^d
	    \exp \left( -c 2^{2k} A_n^4 \right) \\
	&\le \frac{C}{n^{d/2}} \frac{n}{N^d} \sum_{k=0}^\infty 
	    \exp \left( - c 2^{2k} A_n^4 + k d \log 2 + 2d \log A_n \right) \\
	&= o(1) \frac{1}{n^{d/2}} \frac{n}{N^d} \exp \left( - 100 (\log \log n)^2\right).
\end{split}
\end{equation*}
\end{proof}

\begin{lemma}
We have
\[ \sum_{x_1 \not= x_2 \in \mathbb{Z}^d} \mathbf{P} [ A(x_1) \cap A(x_2) ]
   = o(1) \frac{n}{N^d} p_n(y'_{\ell-1}, y'_\ell). \]
\end{lemma}

\begin{proof}
The summand on the left-hand side is bounded above by
\begin{equation*}
\begin{split}
 \P [ A(x_1) \cap A(x_2) ]
 \le \sum_{m_1 + m_2 + m_3 = n} \sum_{\substack{x'_1 \in \bfK+x_1N \\ x'_2 \in \bfK+x_2N}}
	   &\left[ p_{m_1}(y'_{\ell-1}, x'_1) p_{m_2}(x'_1, x'_2) p_{m_3}(x'_2, y'_{\ell}) \right.\\
	   &\left. + p_{m_1}(y'_{\ell-1}, x'_2) p_{m_2}(x'_2, x'_1) p_{m_3}(x'_1, y'_{\ell}) \right].
\end{split}
\end{equation*}
Due to symmetry it is enough to consider the first term inside the summation. 
The estimates are again modelled on the proof of Proposition \ref{prop:weaker_capacity_ub}.

\emph{Case 1. $m_1 + m_2 \ge n/2$ and $|x'_2 - y'_{\ell-1}| \le 2 C_1 \sqrt{n} \sqrt{\log n}$.}
In this case we can use Corollary \ref{cor:x-summation}
%of Proposition \ref{prop:weaker_capacity_ub} 
with $y' = y'_{\ell-1}$ and $y'' = x'_2$ to perform the summation over $x'_1$ and $x_1$ and get the upper bound:
\begin{equation}
 C \frac{n}{N^d} \sum_{m'_1 + m'_2 = n} \sum_{x_2 \in \Z^d} \sum_{x'_2 \in \bfK+x_2N}
   p_{m'_1}(y'_{\ell-1}, x'_2) p_{m'_2} (x'_2, y'_\ell),
\end{equation}
where we have written $m'_1 = m_1 + m_2$ and $m'_2 = m_3$. Using again Corollary \ref{cor:x-summation}, this time with $y' = y'_{\ell-1}$ and $y'' = y'_\ell$ 
%Proposition \ref{prop:weaker_capacity_ub} 
yields the upper bound
\begin{equation}
 C \left( \frac{n}{N^d} \right)^2 p_{n} (y'_{\ell-1}, y'_\ell)
 = o(1) \frac{n}{N^d} p_{n} (y'_{\ell-1}, y'_\ell).
\end{equation}

\emph{Case 2. $m_1 + m_2 \ge n/2$ and $2 C_1 \sqrt{n} \sqrt{\log n} < |x'_2 - y'_{\ell-1}| \le n^{\frac{1}{2}+\eps}$.} We are going to use that $\eps \le 1$, which we can clearly assume.
First sum over all $x'_1 \in \Z^d$ to get the upper bound
\begin{equation}
\label{e:case2-summation}
 C n \sum_{m'_1 + m'_2 = n} \sum_{x_2 \in \Z^d} \sideset{}{'}\sum_{x'_2 \in \bfK+x_2N}
   p_{m'_1}(y'_{\ell-1}, x'_2) p_{m'_2} (x'_2, y'_\ell), 
\end{equation}
where the primed summation denotes the restriction $2 C_1 \sqrt{n} \sqrt{\log n} < |x'_2 - y'_{\ell-1}| \le n^{\frac{1}{2}+\eps}$.
The choice of $C_1$ (recall \eqref{e:C_1-choice}) implies that $p_{m'_1}$ is $o(1/n^{1+3d/2} N^d)$.
Due to the triangle inequality we also have $|y'_\ell - x'_2| > C_1 \sqrt{n} \sqrt{\log n}$.
Using the LCLT for $p_{m'_2}$ we get that
\begin{equation}
 p_{m'_2} (x'_2, y'_\ell)
 \le \frac{C}{(m'_2)^{d/2}} \exp ( - d C_1^2 n \log n / m'_2 )
 \le \frac{C}{n^{d/2}} \exp ( - d C_1^2 \log n )
 \le C p_{n} (y'_{\ell-1}, y'_\ell). 
\end{equation}
Substituting this bound and $p_{m'_1} = o(1/n^{1+3d/2} N^d)$ into \eqref{e:case2-summation} we get 
\begin{equation*}
\begin{split}
    &C n\, o(1) \left(\frac{1}{n \cdot n^{3d/2} \cdot N^d} \right) 
    \sum_{m'_1 + m'_2 = n} \sideset{}{'}\sum_{x'_2}
    p_n(y'_{\ell-1},y'_\ell) \\
    &\qquad \le o(1) \left( \frac{n}{N^d} \right) p_{n}(y'_{\ell-1},y'_\ell) \sideset{}{'}\sum_{x'_2} \frac{1}{(n^{1/2+\eps})^d} \\
    &\qquad = o\left( \frac{n}{N^d} \right) p_{n}(y'_{\ell-1},y'_\ell).
\end{split}
\end{equation*}

\emph{Case 3. $m_1 + m_2 \ge n/2$ and $|x'_2 - y'_{\ell-1}| > n^{\frac{1}{2}+\eps}$.}
Summing over all $x'_1 \in \Z^d$, we get the transition probability 
$p_{m_1+m_2}(y'_{\ell-1}, x'_2)$. This is stretched-exponentially small, and hence this case satisfies the
required bound.

\emph{Case 4. $m_2 + m_3 \ge n/2$.} Due to symmetry, this case can be handled analogously to Cases 1--3.

\end{proof}

\subsection{Proof of Proposition \ref{prop:good_stretches}}
\label{ssec:good_stretches}

\begin{proof}[Proof of Proposition \ref{prop:good_stretches}.]
We start with the proof of the second claim.
We denote the error term in \eqref{e:good_stretches} as $E$, which we claim to satisfy 
$|E| \le E_1 + E_2 + E_3 + E_4$, with 
\begin{equation*}
\label{e:E-terms}
\begin{split}
 E_1
 &= \mathbf{P}_{o} \left[ \left\{\frac{S n}{N^d} < (\sqrt{\log\log n})^{-1} \right\} \cup \left\{\frac{S n}{N^d} > \log N \right\}\right] \\
 E_2
 &= \mathbf{P}_{o} \left[ \text{$\exists \ell: \ 1 \le \ell \le \log N\frac{N^d}{n}$ 
    such that $Y_{\ell n} \in \varphi^{-1}( B(\mathbf{x_0}, N^{\zeta}) )$} \right] \\
 E_3
 &= \mathbf{P}_{o} \left[ \text{$\exists\ t : T \le t < S n$ such that 
    $Y_t \in \varphi^{-1}(\mathbf{K})$} \right]. \\
 E_4
 &= \mathbf{P}_{o} \left[ \text{$\exists \ell: \ 1 \le \ell \le \log N\frac{N^d}{n}$ 
    such that $| Y_{\ell n}-Y_{(\ell-1)n} |> f(n)n^{\frac{1}{2}}$} \right]. 
\end{split}
\end{equation*}
% Lemma 5 E5 bound
%\begin{equation*}
%E_5
% = \mathbf{P} \left[ \text{$\exists \ell: \ 1 \le \ell < \log N\frac{N^d}{n}$ 
%    such that $|Y_{\ell n}-Y_{(\ell-1)n}|\leq N^{\zeta_2}$} \right].
%\end{equation*}
Since $T\leq Sn$, we have 
\begin{align*}
\left|\mathbf{P}_{o} \left[ Y_t \not\in \varphi^{-1}(\mathbf{K}),\, 0 \le t < T \right]
- \mathbf{P}_{o} \left[ Y_t \not\in \varphi^{-1}(\mathbf{K}),\, 0 \le t < Sn \right]\right|
\leq E_3.
\end{align*}
By the Markov property, 
for $(\tau, (y_\ell, \mathbf{y}_\ell)_{\ell=1}^\tau ) \in \mathcal{G}_{\zeta,C_1}$,
\begin{equation*}
\begin{split}
&\prod_{\ell=1}^\tau \mathbf{P}_{y'_{\ell-1}} 
	 \left[ Y_n = y'_{\ell},\, \text{$Y_t \not\in \varphi^{-1}(\mathbf{K})$ for
	 $0 \le t < n$} \right] \\
& = \P_{o}\left[ \parbox{9cm}{$Y_{n \ell}=y'_{\ell}$ for $0 \le \ell \le \tau$; 
	  $Y_t \not\in \varphi^{-1}(\mathbf{K})$ for $0 \le t < \tau n$;
	  $Y_{n\ell } \not\in \varphi^{-1}( B(\mathbf{x_0}, N^{\zeta}) )$ for 
	 $0 < \ell \le \tau$} \right].
\end{split}
\end{equation*}

We denote the probability on the right-hand side by $p(\tau, (y_\ell, \mathbf{y}_\ell)_{\ell=1}^\tau)$.
On the event of the right-hand side, since $y_{\tau}\in D^{c}$, we have $S=\tau$.
We claim that
\begin{equation}
\begin{split}
\label{e:prop_claim}   
    \left| \sum_{( \tau, (y_\ell, \mathbf{y}_\ell)_{\ell=1}^\tau) \in \mathcal{G}_{\zeta,C_1}}p(\tau, (y_\ell,\mathbf{y}_\ell)_{\ell=1}^\tau) - \P_{o}\left[\text{$Y_t \not\in \varphi^{-1}(\mathbf{K})$ for $0 \le t < Sn$}\right]\right|
     \leq E_1 + E_2 + E_4.
\end{split}
\end{equation}

Let $\mathcal{E}_1, \mathcal{E}_2, \mathcal{E}_4$ be the events in the definitions of $E_1, E_2, E_4$ respectively.
Let 
$$A(\tau, (y_\ell,\bfy_\ell)_{\ell=1}^{\tau})
:= \left\{ \parbox{9cm}{$Y_{n \ell}=y'_{\ell}$ for $0 \le \ell \le \tau$; 
	  $Y_t \not\in \varphi^{-1}(\mathbf{K})$ for $0 \le t < \tau n$;
	  $Y_{n\ell } \not\in \varphi^{-1}( B(\mathbf{x_0}, N^{\zeta}) )$ for 
	 $0 < \ell \le \tau$} \right\}.$$
Then, we have
\begin{equation*}
\begin{split}
   &\P_{o}\left[\left\{\text{$Y_t \not\in \varphi^{-1}(\mathbf{K})$ for $0 \le t < Sn$}\right\} \cap \mathcal{E}^{c}_1 \cap \mathcal{E}^{c}_2 \cap \mathcal{E}^{c}_4\right] \\
   &\quad = \sum_{ \frac{N^d}{n\sqrt{\log \log n}} \le \tau \le \frac{N^d \log N}{n}} \P_o \left[\left\{S=\tau\right\} \cap \left\{\text{$Y_t \not\in \varphi^{-1}(\mathbf{K})$ for $0 \le t < \tau n$}\right\} \cap \mathcal{E}^{c}_2 \cap \mathcal{E}^{c}_4 \right] \\
   &\quad = \sum_{( \tau, (y_\ell, \mathbf{y}_\ell)_{\ell=1}^\tau ) \in \mathcal{G}_{\zeta,C_1}} \P_o \left[ A(\tau, (y_\ell,\bfy_\ell)_{\ell=1}^{\tau}) \cap \mathcal{E}^{c}_2 \cap \mathcal{E}^{c}_4 \right].
\end{split}
\end{equation*}

From above, the claim \eqref{e:prop_claim} follows.

The proof of the second claim of Proposition \ref{prop:good_stretches} follows subject to Lemmas \ref{lem:E_2}, \ref{lem:E_3} and \ref{lem:E_4} below that show $E_j \to 0$, for $j= 2, 3, 4$.
We have already shown $E_1 \to 0$ in Lemma \ref{lem:E1_prelim}.

Similarly, the first claim of the Proposition follows from Lemmas \ref{lem:E1_prelim}, \ref{lem:E_2}, \ref{lem:E_4}.
\end{proof}

We bound $E_2, E_3$ and $E_4$ in the following lemmas.

\begin{lemma}
\label{lem:E_2}
We have 
$E_2 \le C \log N \frac{N^{d \zeta}}{n}$ for some $C$.
Consequently, $E_2 \to 0$.
%If $\zeta < \delta/d$, then $E_2\to 0$.
\end{lemma}

\begin{proof}
Since 
%$S\leq \log N\frac{N^d}{n}$ and 
the number of points in $B(\mathbf{x_0}, N^{\zeta})$
is $O(N^{d\zeta})$, and since we are considering times after the first stretch, the random walk is well mixed, 
so by Lemma \ref{lem:mixing_property} the probability to visit any point in the torus is $O(1/N^d)$.
Using a union bound we have
\begin{align*}
E_2 
\leq C \log N\frac{N^d}{n} N^{d\zeta} \frac{1}{N^d} 
= C \log N \frac{N^{d \zeta}}{n} ,
\end{align*}
Since $\zeta < \delta/d$, we have 
\begin{equation*}
    E_2 \to 0, 
    \quad \text{as $N\to\infty$}.
\end{equation*}
%($\zeta < \delta/d$ required condition)
\end{proof}

%Therefore $\zeta < \delta/d$ is a required condition in proving Proposition \ref{prop:good_stretches}.

Before we bound the error term $E_3$, we first introduce the following lemma. 
Let $T^{(i)}_0 = \inf\{t\ge 0: Y^{(i)}_t = 0 \}$, where recall that $Y^{(i)}$ denotes the $i$-th 
coordinate of the $d$-dimensional lazy random walk. 
We will denote by $t_0$ an instance of $T^{(i)}_0$. 

\begin{lemma}
\label{lem:E_3_preliminary}

For all $1 \le i \le d$, for any integer $y$ such that $-t_0 \le y < 0$ and $0 < t_0 \le n$ we have 
\begin{equation*}
\begin{split}
\E_y \left[Y^{(i)}_n \,|\, T^{(i)}_0 = t_0,\, Y_n^{(i)} > 0 \right] 
  &\leq C n^{\frac{1}{2}}\\
\E_y \left[(Y^{(i)}_n)^2 \,|\, T^{(i)}_0 = t_0,\, Y_n^{(i)} > 0 \right] 
  &\leq C n.
\end{split}
\end{equation*}
\end{lemma}

\begin{proof}
Using the Markov property at time $t_0$, we get 
\begin{align*}
\E_y\left[Y^{(i)}_n \,|\, T^{(i)}_0=t_0, Y_n^{(i)} > 0 \right]
&= \E_0 \left[Y_{n-t_0}^{(i)} \,|\, Y_{n-t_0}^{(i)} > 0 \right]
=\frac{\E_0\left[Y^{(i)}_{n-t_0}\mathbf{1}_{\{Y^{(i)}_{n-t_0} > 0\}}\right]}
	{\P_0\left[Y^{(i)}_{n-t_0}>0\right]} \\
&\leq C_0 \left(\E_0\left[(Y^{(i)}_{n-t_0})^2\right]\right)^{\frac{1}{2}}	
\leq C (n-t_0)^{\frac{1}{2}}
\leq C n^{\frac{1}{2}},
\end{align*}
where the third step is due to Cauchy-Schwarz inequality and 
$\P_0\left[Y^{(i)}_{n-t_0}>0\right]\geq c_0$ for some $c_0 > 0$, 
and the second to last step is due to $\E_0\left[(Y^{(i)}_{n-t_0})^2\right] = (n-t_0)/2d$.

We can similarly bound the conditional expectation of $(Y^{(i)}_n)^2$ as follows:
\begin{align*}
&\E_y\left[(Y^{(i)}_n)^2| T^{(i)}_0=t_0, Y^{(i)}_n>0\right]
=\E_0\left[(Y^{(i)}_{n-t_0})^2|Y^{(i)}_{n-t_0}>0\right] \\
&\qquad =\frac{\E_0\left[(Y^{(i)}_{n-t_0})^2\mathbf{1}_{\{Y^{(i)}_{n-t_0} > 0\}}\right]}
	{\P_0\left[Y^{(i)}_{n-t_0}>0\right]} 
\leq C_0 \left(\E_0\left[(Y^{(i)}_{n-t_0})^2\right]\right)
\leq C (n-t_0)
\leq C n. 
\end{align*}
\end{proof}

\begin{lemma}
\label{lem:E_3}
We have $E_3 \to 0$ as $N \to \infty$.
\end{lemma}

\begin{proof}%[Proof (Martingale)]
First we are going to bound the time difference between $T$ and $Sn$.
We are going to consider separately the cases when $Y_T$ is in each face 
of the cube $(-L,L)^d$. Assume that we have $Y^{(i)}_T = L$ for some $1 \le i \le d$.
(The arguments needed are very similar when $Y^{(i)}_T = -L$ for some $1\le i \le d$, and
these will not be given.)

Let us consider the lazy random walk $(Y_t)_{t\ge 0}$ in multiples of $n$ steps. 
Let 
\begin{equation*}
    s_1 
  = \min \{ \ell n : \ell n \ge T \} - T,
\end{equation*}
%denote the first multiple of $n$ greater than or equal to $T$,
and similarly, let 
\begin{equation*}
    s_{r+1}
    = r n + \min \{ \ell n : \ell n \ge T \} - T, \quad r \ge 1.
\end{equation*}
%$s_2$ denote the second multiple of $n$ greater than or equal to $T$, and so on.
We let $M_0 = L-Y^{(i)}_{T+s_1}$ and $M_r = L-Y^{(i)}_{T+s_{r+1}}$ for $r\ge 1$.
We have that $(M_r)_{r\ge 0}$ is a martingale. 
Let $\tilde{S} = \inf \{ r \ge 0: M_r \le 0\}$, and we are going to bound
$\P [ \tilde{S} > N^{\eps_1} ]$ for some small $\eps_1$ that we are going to
choose in the course of the proof.
We are going to adapt an argument in \cite[Proposition 17.19]{LevinPeresWilmer2017}
to this purpose. 
%Some adaptation is needed, as this is not a non-negative 
%martingale as in \cite[Proposition 17.19]{LevinPeresWilmer2017}. 

Define
\begin{equation*}
T_h = \inf\{r\geq 0: \text{$M_r\leq 0$ or $M_r\geq h$} \}, 
\end{equation*}
where we set $h= \sqrt{n}\sqrt{N^{\eps_1}}$. 
% for some small $\eps_1>0$ that we will choose in our proof. 
Let $(\mathcal{F}_r)_{r \ge 0}$ denote the filtration generated by $(M_r)_{r \ge 0}$.
We have 
\begin{equation}
\label{e:var_M_r}
    \Var(M_{r+1} \,|\, \mathcal{F}_r) = n \sigma^2
    \quad \text{for all $r \ge 0$},
\end{equation}
where recall that $\sigma^2$ is the variance of $Y^{(i)}_1$.
%We also have  
%\begin{equation*}
%    M_{T_h} \le h + n \le D h
%    \quad \text{for $D = 2 \sqrt{n}/\sqrt{N^{\eps_1}}$}.
%\end{equation*}

We first estimate $\E [ M_0 \,|\, \tilde{S} > 0 ]$. 
Since $0 \le s_1 < n$, by the same argument as in 
Lemma \ref{lem:E_3_preliminary} we have that 
\begin{equation*}
\begin{split}
\E [ M_0 \,|\, Y^{(i)}_{T+s_1} < L] 
&= \E [ L-Y^{(i)}_{T+s_1} \,|\, L - Y^{(i)}_{T+s_1} > 0 ]
\leq C n^{\frac{1}{2}}.
%\\
%\E_0[M_0]
%&= \E_0[\E_0 [M_0|Y^{(i)}_{s_1}\leq L]]
%\leq C n^{\frac{1}{2}}.
\end{split}
\end{equation*}

%We claim that
%\begin{equation*}
% \P[T_h\geq r \,|\, M_0 ]
% \leq \frac{D h M_0}{n \sigma^2 r} + \frac{C n^{1/2} + Cn}{n \sigma^2 r},  
%\end{equation*}
%where we will chose $r = N^{\eps_1}$.
%
%
%For $\E_0[M_0] \le C\sqrt{n} \le h= \sqrt{n}\sqrt{N^{\eps_3}}$,
%we have that $\{\tau\ge r\}\subseteq \{T_h\ge r\}\cup\{M_{T_h}\ge h\}$, 
%whence
%\begin{equation*}
%\P[\tau\ge r]
%\le \P[T_h\ge r]+\P[M_{T_h}\ge h].
%\end{equation*}
We first bound $\P[M_{T_h}\ge h \,|\, M_0]$. Since $(M_{r\wedge T_h})$ is bounded,
by the Optional Stopping Theorem, we have
\begin{equation*}
\begin{split}
M_0 
&= \E [M_{T_h}|M_0] 
= \E\left[M_{T_h}\mathbf{1}_{\{M_{T_h}\leq 0\}}|M_0\right]
	+\E\left[M_{T_h}\mathbf{1}_{\{M_{T_h}\geq h\}}|M_0\right]\\
&= -m_{-}(h) + \E\left[M_{T_h}\mathbf{1}_{\{M_{T_h}\geq h\}}|M_0\right]\\
&\geq -m_{-}(h) + h\P\left[M_{T_h}\geq h |M_0 \right],
\end{split}
\end{equation*}
where we denote $\E[M_{T_h}\mathbf{1}_{\{M_{T_h}\leq 0\}}|M_0]$ by $-m_{-}(h) \leq 0$ 
and the last step is due to $M_{T_h} \mathbf{1}_{\{M_{T_h} \ge h\}} \ge h \mathbf{1}_{\{M_{T_h} \ge h\}}$.
%Markov's inequality.
Hence, we have 
\begin{align*}
M_0+m_{-}(h) \geq h\P\left[M_{T_h}\geq h|M_0\right].
\end{align*}

We bound $m_{-}(h)$ using Lemma \ref{lem:E_3_preliminary}
\begin{align*}
m_{-}(h)
\leq \max_{y\leq L} \E_y\left[Y^{(i)}_n-L|Y^{(i)}_n > L  \right]
\leq C n^{\frac{1}{2}}.
\end{align*}
%where $y\leq L $ is compared coordinate-wise in this proof. 
Hence, we have 
\begin{align*}
\P[M_{T_h}\geq h \,|\, M_0]
\leq \frac{M_0}{h}+\frac{C n^{\frac{1}{2}}}{h}.
\end{align*}

We now estimate $\P [ T_h \geq r \,|\, M_0 ]$. 
%We bound $m_{-}(h)$ by conditioning on the last step of the random walk that is above $L$, 
%and compute the expected value of the positive increment starting below $L$,
%\begin{align*}
%m_{-}(h)
%&\leq \max_{y\leq L} \E_y\left[Y_n-L|Y_n-L > 0 \right]
%= \max_{y\leq L}\frac{\sum_{x>L}P_n(y,x)(x-L)}{\P_y\left[Y_n > L\right]}\\
%&= \max_{y\leq L}\frac{\sum_{x>L}(x-L)P_n(y,x)}{\sum_{x>L}P_n(y,x)}
%\leq C n^{\frac{1}{2}}.
%\end{align*}
%The maximum attains when $y=L$. The expected amount of movement in $n$ steps is 
%$O(\sqrt{n})$, although conditioned to be positive. $\sqrt{n}$ to check 
%
%We now bound  $\P[T_h\ge r]$.
Let $G_r = M^2_r-hM_r-\sigma^2 nr$. The sequence $(G_r)$ is a martingale by \eqref{e:var_M_r}.
%(? $(G_r)$ is a martingale)
%
%Since $\Var(M_{r+1} | \mathcal{F}_r) = \E[M^2_{r+1}|\mathcal{F}_r] - M_r^2 = n\sigma^2$
%\begin{equation*}
%\begin{split}
%\E[G_{r+1}|\mathcal{F}_r]
%&= \E[M^2_{r+1}|\mathcal{F}_r]-h\E[M_{r+1}|\mathcal{F}_r]-n\sigma^2(r+1)\\
%&= \E[M^2_{r+1}|\mathcal{F}_r]-n\sigma^2 -hM_r-n\sigma^2r 
%= M_r^2-hM_r-n\sigma^2r 
%= G_r.
%\end{split}
%\end{equation*}
We can bound both the `overshoot' above $h$ and the `undershoot' below $0$ by Lemma \ref{lem:E_3_preliminary}.
For the `undershoot' below $0$ we have
\begin{equation*}
\begin{split}
    \E [ (M_{T_h} - h) M_{T_h} \,|\, M_{T_h} \le 0, M_0 ]
    &= \E [ M_{T_h}^2 \,|\, M_{T_h} \le 0, M_0 ] + \E [ -h M_{T_h} \,|\, M_{T_h} \le 0, M_0 ] \\
	&\le C n + C h n^{1/2}. 
\end{split}
\end{equation*}

For the `overshoot' above $h$, write $M_{T_h} =: N_{T_h}+h$, then we have
\begin{equation*}
(M_{T_h}-h)M_{T_h} = N_{T_h}(h+N_{T_h}),
\end{equation*}
Hence
\begin{equation*}
\begin{split}
   \E [ (M_{T_h} - h) M_{T_h} \,|\, M_{T_h} \ge h, M_0 ]
   &= \E [ h N_{T_h} \,|\, N_{T_h} \ge 0, M_0 ] + \E [ N_{T_h}^2 \,|\, N_{T_h} \ge 0, M_0 ]\\
    &\le C h n^{1/2} + C n.  
\end{split}
\end{equation*}

For $r < T_h$, we have $(M_r - h) M_r < 0$.
%if $M_{T_h}\geq h$ or $M_{T_h} \leq 0$, we have 
%\begin{align*}
%M_r^2-hM_r 
%= (M_r-h)M_r
%\leq (D-1)hM_r.
%\end{align*}
%
%When $r=T_h$ and $M_{T_h} \leq 0$,
%\begin{align*}
%M_r^2-hM_r 
%&\leq \max_{y\leq L}\E_y\left[(Y_n-L)^2 | Y_n-L>0 \right]
%	+h\max_{y\leq L} \E_y\left[Y_n-L|Y_n-L > 0 \right]\\
%&= \max_{y\leq 0}\E_y\left[Y^2_n | Y_n>0 \right]
%	+h\max_{y\leq 0}\E_y\left[Y_n | Y_n>0 \right] 
%\leq Cn + hCn^{\frac{1}{2}}
%\end{align*}
%
Therefore, we have
\begin{align*}
\E\left[M^2_{r\wedge T_h}-hM_{r\wedge T_h}| M_0\right]
%&\leq (D-1)hM_0+\max_{y\leq 0}\E_y\left[Y^2_n | Y_n>0 \right]
%	+h\max_{y\leq 0}\E_y\left[Y_n | Y_n>0 \right] \\
&\leq C h n^{1/2} + C n. 
\end{align*}

Since $(G_{r\wedge T_h})$ is a martingale
\begin{equation*}
\begin{split}
-hM_0 \le G_0 
\le \E[ G_{r\wedge T_h}| M_0]
&= \E[M^2_{r\wedge T_h}-hM_{r\wedge T_h}| M_0]-\sigma^2 n\E[r\wedge T_h| M_0]\\
&\le Cn^{\frac{1}{2}}h+Cn-\sigma^2 n \E[r\wedge T_h| M_0].
\end{split}
\end{equation*}
We conclude that $\E[r\wedge T_h \,|\, M_0 ] \leq \frac{h(M_0+Cn^{\frac{1}{2}})+Cn}{\sigma^2 n}$.
Letting $r\to\infty$, by the Monotone Convergence Theorem, 
\begin{equation*}
\E[T_h \,|\, M_0 ] \leq \frac{h(M_0+Cn^{\frac{1}{2}})+Cn}{\sigma^2 n},
\end{equation*}
where $h= \sqrt{n}\sqrt{N^{\eps_1}}$. This gives
\begin{equation*}
\begin{split}
\P[T_h> N^{\eps_1}|M_0] 
\le \frac{1}{N^{\eps_1}}\left[\frac{\sqrt{n}\sqrt{N^{\eps_1}}M_0+Cn\sqrt{N^{\eps_1}}+Cn}{\sigma^2 n}\right].
\end{split}
\end{equation*}
Taking expectations of both sides, we have 
\begin{equation*}
\begin{split}
\P[T_h>N^{\eps_1}] 
&\le \frac{1}{N^{\eps_1}}\left[\frac{\sqrt{n}\sqrt{N^{\eps_1}}\E M_0
+Cn\sqrt{N^{\eps_1}}+Cn}{\sigma^2 n}\right]\\
&= \frac{\E M_0}{\sigma^2\sqrt{n}\sqrt{N^{\eps_1}}}+\frac{C}{\sigma^2\sqrt{N^{\eps_1}}}
+\frac{C}{\sigma^2N^{\eps_1}}
\le \frac{C}{\sqrt{N^{\eps_1}}}.
\end{split}
\end{equation*}

%\begin{align*}
%\frac{D\sqrt{n}\sqrt{N^{\eps_3}}(M_0+Cn^{1/2})+Cn}{\sigma^2t}
%\leq \frac{D\sqrt{n}\sqrt{N^{\eps_3}}(M_0+Cn^{1/2})+Cn}{CnN^{\eps_3}}.
%\end{align*}

Combining the above bounds, we get
\begin{align*}
\P[\tilde{S} > N^{\eps_1}]
&\leq 
\P[ M_{T_h} \ge h ] + \P[ T_h > N^{\eps_1} ] 
\le \frac{\E [ M_0 ]}{h}
	+\frac{C n^{\frac{1}{2}}}{h}+ \frac{C}{\sqrt{N^{\eps_1}}}
\leq \frac{C}{\sqrt{N^{\eps_1}}}.
\end{align*}

We now bound the probability that a copy of $\bfK$ is hit between times 
$T$ and $s_{N^{\eps_1}}$.

We first show that the probability that the lazy random walk on the torus is in the ball 
$B(\bfx_0, N^\zeta)$ at time $T$ goes to $0$. Indeed, we have
\begin{align*}
&\P_{o}\left[Y_{T}\in \varphi^{-1}\left(B(\bfx_0, N^{\zeta})\right)\right]
= \sum_{y' \in\partial (-L,L)^d \cap \varphi^{-1}(B(\bfx_0, N^{\zeta}))}
	\P_{o}\left[Y_{T}=y'\right]\\
&\qquad \leq C N^{\zeta(d-1)} \frac{L^{d-1}}{N^{d-1}} \frac{C}{L^{d-1}}
= C N^{(\zeta-1)(d-1)},
\end{align*}
where we have $\zeta<\delta/d < 1$, so the last expression goes to $0$.
Here we used that $\P_{o} [ Y_T = y' ] \le C/L^{d-1}$, for example using
a half-space Poisson kernel estimate \cite[Theorem 8.1.2]{LawlerLimic2010}.
As for the number of terms in the sum, we have that there are $C L^{d-1}/N^{d-1}$ copies of the torus within $\ell_{\infty}$ distance $\le N$ of the boundary $\partial (-L,L)^d$. 
Considering the intersection of the union of balls $\varphi^{-1}(B(\bfx_0, N^{\zeta}))$ and the boundary
$\partial (-L, L)^d$, the worst case is that within a single copy of the torus the intersection has size at most $C N^{\zeta(d-1)}$.

%We know the probability that the lazy random walk exits at any particular point of the box 
%$[-L, L)^d$
%
%(ref need to find, a ball is showing in LawlerLimic book,
%rather than a cube, similarly to the ball case. local limit theorem) (??)
%\begin{align*}
%\P_{o}\left[Y_{T}=x\right]
%\leq \frac{C}{N^{d-1}},
%\end{align*}
%for all $x\in \partial [-L, L)^d$.

Condition on the location $y'$ of the walk at the exit time $T$.
For $y'\notin \varphi^{-1}\left(B(\bfx_0, N^{\zeta})\right)$ %for $\bfx_0\in\bfK$, 
we first bound the probability of hitting $\bfK$ between the times between $0$ and $s_2$.
After time $s_2$, the random walk is well mixed, and we can apply a simpler 
union bound.

We thus have the upper bound
\begin{align*}
\sum_{t=0}^{s_2}\sum_{x'\in\varphi^{-1}(\bfK)}p_t(y',x')
\leq \P\left[\max_{0\leq t\leq s_2}|Y^{(i)}_{t}-y'| \geq n^{\frac{1}{2}+\eps}\right] 
	+\sum_{\substack{x'\in\varphi^{-1}(\bfK)\\|x'-y'|\leq n^{\frac{1}{2}+\eps}}}G(y',x').
\end{align*}
The first term is stretched-exponentially small due to the Martingale maximal inequality \eqref{e:maximal_ineq}.
The Green's function term is bounded by Lemma \ref{lem:Green-bnd}(iii).

After time $s_2$, by Lemma \ref{lem:mixing_property}, we have that
\begin{align*}
\sum_{t=s_2}^{s_{N^{\eps_1}}}\P_y\left[Y_t\in\varphi^{-1}(K)\right]
\leq n\cdot N^{\eps_1}|\bfK|\frac{C}{N^d}
= C\, N^{\delta+\eps_1-d}.
\end{align*}

Therefore, combining the above upper bounds, we have the required result. 
\begin{align*}
E_3\leq C \cdot N^{-\frac{\eps_1}{2}}
  +C \cdot N^{\delta-d+2 \delta \eps} + C\cdot N^{\delta-d+\eps_1} \to 0, 
\quad \text{as $N\to\infty$,}
\end{align*}
if $\eps$ and $\eps_1$ are sufficiently small.
\end{proof}

\begin{lemma}
\label{lem:E_4}
We have 
$E_4 \le C e^{-cf(n)^2}\frac{N^d \log N}{n}$ for some $C$.
Consequently, 
there exists $ C_1$ such that if $f(n)\geq C_1\sqrt{\log N}$, then $E_4\to 0$.
\end{lemma}

\begin{proof}
By the Martingale maximal inequality \eqref{e:maximal_ineq}, 
we have that 
\begin{align*}
E_4 \le C e^{-cf(n)^2}\frac{N^d \log N}{n}.
\end{align*}

Taking say $C_1 > \sqrt{d/c}$ implies that if $f(n)\geq C_1\sqrt{\log N}$, we have
\begin{align*}
E_4\to 0, \quad\text{as $N\to\infty$}.
\end{align*}

\end{proof}

%(The following lemma may not need, but modify it later)
%\begin{lemma}
%We have 
%$E_5 \le C \frac{N^d\log N}{n} N^{d(\zeta_2-\frac{\delta}{2})}$ for some $C$.
%If $\zeta_2 < \delta/d$, then $E_5\to 0$.
%\end{lemma}
%
%\begin{proof}
%We know that
%\begin{align*}
%\P_{o}\left[|Y_n|\leq N^{\zeta_2}\right]
%\leq \frac{C}{n^{d/2}}N^{d\zeta_2}
%= C\ N^{d(\zeta_2-\frac{\delta}{2})}. 
%\end{align*}
%
%Since $\delta >2$, we have
%\begin{align*}
%\frac{N^d}{n}C\ N^{d(\zeta_2-\frac{\delta}{2})}
%= C\ N^{d(1+\zeta_2-\frac{\delta}{2})-\delta}
%\to 0,
%\end{align*}
%if $\zeta_2\leq \delta/d$.
%\end{proof}

\subsection{Proof of Proposition \ref{prop:number_l}}
\label{ssec:number_l}

\begin{proof}

By Martingale maximal inequality \eqref{e:maximal_ineq} used in the second step we have
\begin{equation*}
\begin{split}
 \P_{y'_{\ell-1}}[|Y_n-y'_{\ell-1}|>\sqrt{n}(10\log\log n)]
 &=\P_{0}[|Y_n|>\sqrt{n}(10\log\log n)]\\
 &\le \exp(-c 100(\log\log n)^2).
\end{split}
\end{equation*}
Hence we have
\begin{equation*}
\begin{split}
 &\E \left[\# \left\{ 1 \le \ell \le \frac{N^d}{n}C_1\log N : |Y_{n\ell} - Y_{n(\ell-1)}| > 10 \sqrt{n} \log\log n \right\} \right] \\
 &\qquad \le \frac{N^d}{n} C_1 \log N \exp(-c(\log\log n)^2)
 \le \frac{N^d}{n} C \exp (-(c/2) (\log\log n)^2 ).
\end{split}
\end{equation*}
By Markov's inequality, it follows that 
\begin{equation*}
\begin{split}
 &\P \left[ \# \left\{ 1 \le \ell \le \frac{N^d}{n}C_1\log N : |Y_{n\ell} - Y_{n(\ell-1)}| > 10 \sqrt{n} \log\log n \right\} 
    \ge \frac{N^d}{n} (\log \log n)^{-1} \right] \\
 &\qquad \le \frac{\frac{N^d}{n} C \exp (-(c/2) (\log\log n)^2 )}{\frac{N^d}{n} (\log \log n)^{-1}}
 \le C \frac{\exp (-(c/2) (\log\log n)^2 )}{(\log \log n)^{-1}}
 \to 0,
\end{split}
\end{equation*}
as $N\to\infty$.

\end{proof}

\textbf{Acknowledgements.} We thank two anonymous referees for their constructive criticism.
The research of Minwei Sun was supported by an EPSRC doctoral training grant to the University of Bath 
with project reference EP/N509589/1/2377430.

%\clearpage

\bibliographystyle{plain}
\bibliography{torus_argument}

\noindent\small
Antal A.~J\'arai and Minwei Sun\\
Address: Department of Mathematical Sciences, University of Bath, Claverton Down, Bath, BA2 7AY, United Kingdom\\
Email: \texttt{A.Jarai@bath.ac.uk}, \texttt{ms2271@bath.ac.uk}

\end{document}